\documentclass[11pt,reqno]{amsart}
\usepackage{setspace}
\title[Generators of Residual Intersections]{Residual Intersections are Koszul-{\rm Fitt}ing ideals}
\author[Vinicius Bou\c{c}a]{Vinicius Bou\c{c}a}
\address{$^{2}$  Instituto de Matem\'{a}tica,
Universidade Federal do Rio de Janeiro,  Brazil. }
\email{bouca.vinicius@gmail.com and vbouca@im.ufrj.br }
\author[S.Hamid Hassanzadeh]{S. Hamid Hassanzadeh}
\address{$^{2}$  Instituto de Matem\'{a}tica,
Universidade Federal do Rio de Janeiro,  Brazil. }
\email{hamid@im.ufrj.br and  hassanzadeh.ufrj@gmail.com}
\usepackage{color}
\usepackage[pagebackref]{hyperref}
\usepackage{enumitem}
\usepackage{mathtools}
\usepackage{amsfonts}
\usepackage{amsmath}
\usepackage{amsthm}
\usepackage[dvipsnames]{xcolor}
\usepackage{graphicx}
\input{xy}
\xyoption{all}
\usepackage{mathtools}
\headsep=1cm \numberwithin{equation}{section}
\newtheorem{defi}{Definition}[section]
\newtheorem{thm}[defi]{Theorem}
\newtheorem{exa}[defi]{Example}
\newtheorem{prop}[defi]{Proposition}
\newtheorem{conj}[defi]{Conjecture}

\newtheorem{cor}[defi]{Corollary}
\newtheorem{lemma}[defi]{Lemma}
\newtheorem{nota}[defi]{Notation}
\newtheorem{rmk}[defi]{Remark}
\newcommand{\double}{E^{\bullet,\bullet}}
\newcommand{\doublel}{E^{p,q}}
\newcommand{\dvert}{d_v}
\newcommand{\dhor}{d_h}

\newcommand{\fm}{\mathfrak{m}}

\newcommand{\fp}{\frak{p}}

\newcommand{\fa}{\frak{a}}
\newcommand{\ft}{\frak{t}}

\DeclareMathOperator{\sgn}{sgn}\DeclareMathOperator{\Hom}{Hom}  \DeclareMathOperator{\coker}{Coker}
\DeclareMathOperator{\Ker}{Ker} \DeclareMathOperator{\Ann}{Ann} \DeclareMathOperator{\Sym}{Sym}    
  
\DeclareMathOperator{\depth}{depth}  \DeclareMathOperator{\Ht}{ht} \DeclareMathOperator{\Tot}{Tot} 
\DeclareMathOperator{\grade}{grade}    
\DeclareMathOperator{\E}{\textit{E}}
\DeclareMathOperator{\Z}{\textit{Z}}\DeclareMathOperator{\Kitt}{Kitt}\DeclareMathOperator{\Fitt}{Fitt}
 \DeclareMathOperator{\SD}{SD} \DeclareMathOperator{\SDC}{SDC}

 \def\ff{{\bf f}}\def\aa{{\bf a}}
  
\def\ag{{\bm \gamma}}

\def\ag{\boldsymbol \gamma}

\begin{document}
\maketitle

\begin{abstract}One  describes  generators of disguised residual intersections in any commutative Noetherian rings.  It is shown  that,  over Cohen-Macaulay rings, the disguised residual intersections and algebraic residual intersections are the same, for ideals with sliding depth. This coincidence provides structural results for algebraic residual intersections in a quite general setting. It is shown how the  DG-algebra structure of Koszul homologies affects the determination of  generators of residual intersections.  In the midway, it is shown  that the Buchsbaum-Eisenbud family of complexes can be  derived  from the Koszul-\v{C}ech spectral sequence. This interpretation of Buchsbaum-Eisenbud families has a crucial rule to establish the above results.
\end{abstract}
\tableofcontents
\section{Introduction}

Residual intersections have a long history in algebraic geometry which goes back to Cayley-Bacharach theory, or at least in the middle of the nineteenth century to Chasles \cite{Ch}  who counted the number of conics tangent to a given conic (See Eisenbud's talk \cite{EisT}  or Kleiman\cite{K} for the historical introduction). 

In intersection theory, the concept of ``residual schemes" and ``residual intersections" are the base of the residual intersection formula of Fulton, Kleiman, Laksov, and MacPherson (or others such as  \cite{Wu}) which describes  a decomposition of the refined intersection product  \cite[Corollary 9.2.3]{Fu}. The theory became part of commutative algebra in the work of Artin and Nagata \cite{AN} wherein Artin and Nagata defined ``Algebraic Residual Intersections" to study the``double point locus" of maps between schemes of finite type over a field.  Although  Fulton's definition \cite[Definition 9.2.2]{Fu} is not the same as the following definition of the algebraic residual intersection, there is a tight relation between them in the affine case (see the introduction of \cite{HN}). 

From the commutative algebra point of view, the theory of residual intersections is a vast generalization of the ``Linkage Theory" of Peskine and Szpiro \cite{PS}. The family of $s$-residual intersections contains  determinantal ideals and, obviously, complete intersections of codimension $s$.  Precisely, 
if $R$ is a commutative   Noetherian  ring,  $I$  an  ideal of
grade $g$  and $s \geq g$  an integer, then
\begin{itemize}
\item{An {\it (algebraic) $s$-residual intersection} of $I$ is a proper ideal $J$ of $R$ such that $\Ht(J)\geq s$ and $J=(\fa:_R I)$ for some ideal $\fa\subset I$ which is generated by $s$ elements.}
\item{A {\it geometric $s$-residual intersection} of $I$ is an algebraic $s$-residual intersection $J$ of $I$ such that $\Ht(I+J)\geq s+1.$}
\item{ An {\it arithmetic $s$-residual intersection} of $I$  is an algebraic $s$-residual intersection such that $\mu_{R_{\fp}}((I/\fa)_{\fp})\leq 1$ for all prime ideal $\fp \supseteq (I+J)$ with $\Ht(\fp)\leq s$. ( $\mu$ denotes the minimum number of generators.)}
\end{itemize}

Since 1983, after the work of Huneke \cite{Hu}, the theory became stronger and stronger due to a series of works of Huneke, Ulrich, Kustin, Chardin, Eisenbud,  and  others. Due to the ubiquity of residual intersections, the theory has attained more attention in the recent years (e.g. \cite{CEU}, \cite{CEU1}, \cite{CNT}, \cite{EU}, \cite{EisT} and \cite{HHU}).   However, there are still many basic and mysterious  properties of $s$-residual intersections which are not established. One of the very basic tasks is
\begin{center}
{\it To determine the generators of an $s$-residual intersection.}
\end{center}
Even from the computational point of view, calculating  the generators of $J=\fa:I$ using elimination, is  expensive. Theoretically, there are few cases for which the set of generators of $J$ can be described, namely,
\begin{itemize}
\item{If $R$ is Gorenstein $I$ is  perfect and $s=g$,  \cite{PS}; }
\item{If $R$ is Gorenstein and $I$ is a complete intersection,    \cite[Theorem 4.8]{BKM} and  \cite[Theorem 5.9(i)]{HU}; }
\item{If $R$ is Cohen-Macaulay(CM) and $I$ is a perfect ideal of height 2,   \cite{Hu},\cite{KU} and \cite[Theorem 1.1]{CEU};}
\item{If $R$ is Gorenstein and $I$ is  perfect Gorenstein ideal of height 3,  \cite[Section 10]{KU}.  }
\item{If $R$ is Gorenstein, $I$ is Gorenstein licci, generically a complete intersection ideal and $s=g+1$,   \cite[Corollary 2.18]{KMU};}
\end{itemize}

In this paper, we describe the generators of residual intersections if $s\leq g+1$, or if the ideal $I$  satisfies the sliding depth condition $\SD_1$, Definition \ref{dsd}. This wide range contains all of the above cases. The turning point is that, instead of looking at the module structure of Koszul homologies of $I$, we study the {\it Differential Graded Algebra} structure of $H_{\bullet}(I)$. 
There is a new ideal which comes into being. We call this ideal the {\it Koszul-Fitting} ideal associated to $I$ and $\fa$, and we  denote it by $\Kitt(\mathfrak a,I)$\footnote{As in the fictions, this acronym  can be used for two different entities.}. $\Kitt(\mathfrak a,I)$  is defined as follows.

Let $R$ be a commutative Noetherian ring, $I = (f_1,\cdots ,f_r)\subseteq R$ an ideal, $\mathfrak a = (a_1,\cdots ,a_s)\subseteq I$. Let $\Phi=(c_{ij})$ be an $r\times s$ matrix in $R$ such that $(a_1,\cdots ,a_s)=(f_1,\cdots ,f_r)\cdot\Phi$.  Let  $K_\bullet= R\langle e_1,\cdots ,e_r;\partial(e_i)=f_i\rangle$ be the Koszul complex equipped with the  structure of the differential graded algebra. Let $\zeta_j = \sum_{i=1}^r c_{ij}e_i,1\leq j\leq s$, $\Gamma_\bullet=R\langle\zeta_1,\cdots ,\zeta_s\rangle$ the algebra generated by the $\zeta$'s, and $Z_\bullet$ be the algebra of Koszul cycles. Then we look at the elements of degree $r$  in the sub-algebra of $K_{\bullet}$ generated by $\Gamma_\bullet$ and $ Z_\bullet$ and define 

$${\rm \Kitt}(\fa,I):= \langle\Gamma_\bullet \cdot Z_\bullet\rangle _r\subseteq K_r = R.$$


We summarize some of our results in the following theorem.
\begin{thm}\label{Main} With the above  token, we have
\begin{enumerate}
\item{$\Kitt(\mathfrak a,I)$ depends only on the ideals $\fa$ and $I$ and not on the generators or representative matrix.}
\item{$\Kitt(\mathfrak a,I)$ is, indeed,  the disguised residual intersection introduced in \cite{HN}; hence $\Kitt(\mathfrak a,I)\subseteq \fa:_RI$, they have the same radical and $\Kitt(\mathfrak a,I)=\fa:_RI$ if $\mu(I/\fa)\leq 1$.}
\item{$ \Kitt(\mathfrak a,I)= \fa+<\Gamma_\bullet \cdot \tilde{H}_\bullet>_r$ where $\tilde{H}_\bullet$ is the sub-algebra of  $K_{\bullet}$ generated by the representatives of Koszul homologies. However $ \Kitt(\mathfrak a,I)\supseteq \Fitt_0(I/\fa)$. }
\item{If $R$ is CM, $J=\fa:_RI$ is an $s$-residual intersection and $I$ satisfies SD$_1$ then $\Kitt(\mathfrak a,I)=J$ and it is a CM ideal.} 
\end{enumerate}
\end{thm}

So that not only the  Cohen-Macaulayness, but  also the structure of the residual intersections is  determined by the above theorem.

   We briefly recall the history behind this theorem.  In 1983, Huneke \cite{Hu} showed that the G$_s$ condition defined by Artin and Nagata \cite{AN}, in 1972, is not enough to prove the CM-ness of residual intersections in a CM ring (indeed the main theorem in Artin-Nagata paper was wrong although the applications were not); so that Huneke defined Strongly Cohen-Macaulay (SCM) ideals  to show that residual intersections of SCM+G$_s$ ideals are Cohen-Macaulay. SCM condition then relaxed to Sliding Depth condition  in \cite{HVV}. In their 1988 landmark paper, \cite{HU}, Huneke and Ulrich asked the question of whether  G$_s$ condition  is at all needed to prove that SCM ideals have CM residual intersection? Following earlier work of Chardin and Ulrich \cite{CU},  Hassanzadeh \cite{Ha} answered this question, affirmatively, for arithmetic residual intersections.  The answer is based on  new construction which is called {\it Disguised Residual Intersection}\footnote{Disguised residual intersections are limit terms in a particular Koszul-\v{C}ech-$\mathcal{Z}_{\bullet}$ spectral sequence,\cite{HN}. } in \cite{HN}. Disguised residual intersections are CM under sliding depth conditions and coincide with the algebraic residual intersections in many different cases; so that Hassanzadeh and Naeliton in \cite[Conjecture 5.9]{HN} conjectured that the disguised residual intersection is the same as the residual intersection for ideals with sliding depth. CM-ness of algebraic residual intersection for ideals with sliding depth was finally proved  in \cite{CNT} in 2018. Theorem \ref{Main}(4) proves  \cite[Conjecture 5.9]{HN}, in particular.

The whole paper is designed to prove Theorem \ref{Tmain}(Theorem \ref{Main}(4)). The idea of the proof is to  use  reduction modulo regular sequences which reduces the problem to the case where $I$ has height $2$; one then  uses the important duality result of \cite{CNT} to show that the disguised residual and algebraic residual are the same. The byproduct is that we have already  taken off the mask of the disguised residual intersection in Theorem \ref{Tgensdisguised2}(Theorem \ref{Main}(2)); so that we have a description of the generators of the residual intersection. 

The paper is divided into $5$ sections.   To understand the structure of disguised residual intersections,  one needs to explicitly determine the maps in any pages of the  Koszul-\v{C}ech-$\mathcal{Z}_{\bullet}$ spectral sequence. This part is done in the first section where we define $r$-liftable elements and determine when an element in any pages of spectral sequences is liftable, Theorems \ref{Tliftctrit} and \ref{Timd}.   This study led us to a rediscovery  of the Buchsbaum-Eisenbud family of complexes which in turn contains the Eagon-Northcott and Buchsbaum-Rim complexes. In section $3$,  we  show  that the structure of  Buchsbaum-Eisenbud family can be simply explained by a Koszul-\v{C}ech spectral sequence; so that the complex-structure, acyclity and many other properties  are natural consequences of the convergence of the spectral sequence. The crucial point is the connecting maps, Theorem \ref{Tlifting}, which have been  determined in section $3$. In order not to diverge from the main goal we don't add any  corollaries in this section.  In section $4$, we first recall some of the basic definitions including ``disguised residual intersection" and determine the generators of disguised residual intersections in quite general setting, Theorem \ref{Tgensdisguised2}. This is
due to the structure of the Eagon-Northcott complex and its coincidence with a strand Koszul-\v{C}ech spectral sequence (the fact developed in section $3$). Subsection \ref{SSindependent} is devoted to showing that the construction of the disguised residual intersection does not depend on the choice of the generators. We prove Theorem \ref{Main}(3) in \ref{TkittH}. Another structural result in this section is Theorem \ref{Tkittspecial} wherein it is shown that  disguised residual intersections specialize.
Finally, in section $5$ we collect several corollaries and applications, more notably, Theorem   \ref{Tmain}   proves Conjecture \cite[Conjecture 5.9]{HN} to a certain extent. We finish the paper by gathering  some interesting corollaries of the main theorems.

\section{Maps in  Spectral Sequences}\label{SSS}
To explain the structure of the disguised residual intersection, we need a concrete explanation of the maps in spectral sequences. Unfortunately there does not exist such an explanation in the literature although the facts may be known to the experts. Here we follow the notations in \cite{Eis}. 	
	
More details may be found in the first author's Ph.D. thesis \cite{B}
	\begin{defi}\label{Dtot}
		Let $R$ be a commutative ring and  $(\double,\dvert,\dhor)$ be a first quadrant double complex of $R$-modules. The total complex of this bi-complex is a graded module with $k$'th component
		 $$\Tot(\double)^k =\bigoplus_{p+q=k}\doublel$$ coming with a degree one differential $d$ given by $$d(m) = (\dvert(m),\dhor(m)) \in E^{p+1,q}\oplus E^{p,q+1} ~\text{for~all~} m\in \doublel$$
		 
The $p$-th vertical filtration of $\Tot(\double)$ is defined by 
 $$(_{ver}\Tot(\double)^p)^k = \bigoplus\limits_{i\geq p}E^{i,k-i}$$
		Similarly, we can define the $p$-th horizontal filtration by putting
		$$(_{hor}\Tot(\double)^p)^k =  \bigoplus\limits_{i\geq p}E^{k-i,i}$$
	\end{defi}

	In what follows, we'll be working with the totalization of a first quadrant spectral sequence, as in Definition
 \ref{Dtot}, filtered by the vertical filtration. Let $F_{\bullet} = \bigoplus\limits_p (_{ver}\Tot(\double)^p)$ and  
    $$^0E^{\bullet,\bullet}=\bigoplus\limits_p\frac{(_{ver}\Tot(\double)^p)}{(_{ver}\Tot(\double)^{p+1})}\simeq \bigoplus\limits_p \E^{p,\bullet}.$$
      These modules are bi-graded: one degree, given by $p$, is given by the filtration, and the other, given by $k$, is the grading coming from the complex $\Tot(\double)$. Let $q = k-p$.  Consider the exact sequence (we remove the subscript $ver$ for the rest of this section.)
     $$0\rightarrow\bigoplus\limits_p \Tot(\double)^{p+1}\xrightarrow{i}\bigoplus\limits_p \Tot(\double)^p\xrightarrow{\pi}\bigoplus\limits_p \Tot(\double)^p/\Tot(\double)^{p+1}\rightarrow0$$
	
	It's clear that $^0E^{p,q} = \doublel$. Moreover, if we fix the degree $p$, then $^0E^{p,\bullet}$, with the induced differential $d_0$, is just the complex $E^{p,\bullet}$ with the differential $\dvert$. So, from the above complex, we get the exact couple $(H(F),H(E),i^*,\pi^*,\delta)$. Both $H(F),H(E)$ are also bi-graded.  Setting   $k=p+q$, we have 
	\begin{equation}\label{SpectralRule}
	\end{equation}
	\begin{enumerate}
		\item $H(F)^{p,q} = H^k(\Tot(\double)^p)$
		\item $H(E)^{p,q} = H^q(E^{p,\bullet})$
		\item The map $i^*$ takes a cohomology class $(a_{p+1},\cdots ,a_k)\in H^k(\Tot(\double)^{p+1})$ to the cohomology class $(0,a_{p+1},\cdots ,a_k)\in H^k(\Tot(\double)^{p})$ and has bi-degree $(-1,1)$
		\item The map $\pi^*$ takes a cohomology class $(a_p,a_{p+1},\cdots ,a_k)\in H^k(\Tot(\double)^{p})$ to the cohomology class of $a_p\in H^q(E^{p,\bullet})$. It has bi-degree (0,0)
		\item The map $\delta$ is a snake-lemma-like map which takes an element $m\in H^q(E^{p,\bullet}) $ to the element $(\dhor(m),0,\cdots ,0)\in H^{k+1}(\Tot(\double)^{p})$ and has bi-degree $(1,0)$
	\end{enumerate}

	So, if we put $F^{(1)} = H(F)$, $^1E=H(E)$ and $d_1 =\pi\circ\delta$ and iterating the process, we get exact couples ($F^{(r)},~^rE,i^{*(r)},(\pi)^{(r)},\delta^{(r)} $) with  differentials $d^r= (\pi)^{(r)}\circ\delta^{(r)}$ on $^rE$.
		The spectral sequence $(^rE^{\bullet,\bullet},d^r)$ constructed above is the spectral sequence associated to the vertical filtration.

We denote the kernel and the image of $d^r$ by $^rZ^{\bullet,\bullet}$ and $^rB^{\bullet,\bullet}$ respectively.	

\medskip
The following definition is new to the context.

	\begin{defi}\label{D}
		Let $m\in \doublel$. We say that $m$ is an $r-$liftable element, for some integer $r$,  if $\dvert(m) = 0$ and there is a sequence of elements $(m,a_1,\cdots ,a_r)$ such that \begin{enumerate}
			\item $\dhor(m) = \dvert(a_1)$
			\item $\dhor(a_i) = \dvert(a_{i+1})$ for every $0 \leq i \leq r-1$.
		\end{enumerate}
	$a_r$ is called an $r$-th lift of $m$ and   the sequence $(m,a_1,\cdots ,a_r)$ is called an $r$-lift sequence of $m$.
	\end{defi}
	\begin{rmk}\label{Rlifthomology}
	\rm	Let $m\in E^{p,q}$, $k=p+q$ and $(m,a_1,\cdots ,a_r)$ be an $r$-lift sequence with $\dhor(a_r)=0.$ Then the sequence $(\textbf{0},m,-a_1,a_2,\cdots ,(-1)^ra_r,\textbf{0})$ gives a cohomology class in $H^k(\Tot(\double)$. Conversely, if $(\textbf{0},m,a_1,\cdots ,a_r,\textbf{0})$ gives a cohomology class in $H^k(\Tot(\double)$, then $$\dvert(m)=0, \dvert(a_1) = -\dhor(m), \dvert(a_{i+1})= -\dhor(a_i).$$ Therefore $(m,-a_1,a_2,\cdots ,(-1)^ra_r)$ is an $r$-lift sequence for $m$.   
. 	\end{rmk}
 We now describe the differentials $d^r$ on the $r$-th page of the spectral sequence.
 \begin{prop}\label{Tkerd}
 	Let $m\in \E^{p,q}$ be an $r$-liftable element, and $(m,a_1,a_2,\cdots ,a_r)$ an $r$-lift sequence. Then $m \in~ ^iZ^{p,q}$ for all $0\leq i\leq r$ and $d^{r+1}(m) $ is the class of $\dhor(a_r)$ in $^{r+1}E^{p+r+1,q-r}.$
 \end{prop}
\begin{proof} We need to follow the rules of definition  of the spectral sequence by an exact couple. It is easy to see that   $d^0=\dvert$ and $d^1=\dhor$.

We start with the case $r = 1$. Suppose that $m$ is $1$-liftable and let $(m,a)$ be a $1$-lift sequence. Considering the map $\delta$ in (\ref{SpectralRule})
 \begin{equation}\label{Edelta}
	\delta(m) = (\dhor(m),0,\cdots ,0) \in H^{k+1}(\Tot(E^{\bullet,\bullet})).
\end{equation}

	As $(m,a)$ is a $1$-lift sequence, we have $\dhor(m)=-\dvert(a)$. Moreover, $(-\dvert(a),0,\cdots ,0)$ and $(0,\dhor(a),\cdots ,0)$ are cohomologous in $H^{k+1}(\Tot(E^{\bullet,\bullet}))$. To calculate $d^1(m)$, we have to project $(0,\dhor(a),\cdots ,0)$ onto the first coordinate. Therefore $m\in \ker d^1$. To calculate $d^2(m)$ we must calculate $\pi^{(1)}(0,\dhor(a),\cdots ,0)$. By definition, we must take preimage by $i^*$ once and then project onto the first coordinate. As \begin{equation}
	(i^*)^{-1}((0,\dhor(a),\cdots ,0)) = (\dhor(a),\cdots ,0),
	\end{equation}
	$d^2(m)$ is the class of $\dhor(a)$ in $^2E^{p+2,q-1}.$
		
	Suppose now that $m$ is $r-$liftable and that $(m,a_1,\cdots ,a_r)$ is an $r-$lift sequence. Again, applying $\delta$ to $m$, we have the same equation as in (\ref{Edelta}). 
	Thence  $$
	(\dhor(m),0,\cdots ,0)=(-\dvert(a_1),0,\cdots ,0)\sim(0,\dhor(a_1),0,\cdots ,0) = (0,-\dvert(a_2),0,\cdots ,0)\sim$$
	\begin{equation}
	(0,0,\dhor(a_2),0,\cdots ,0) = \cdots \sim(\overbrace{0,\cdots ,0}^r,\dhor(a_r),0,\cdots ,0),
	\end{equation}
	where $\sim$ means homologous. To apply $d^{i}$, we must take  $(i-1)$ times the preimage by $i^*$  and then project onto the first coordinate. Thus, if $i\leq r$, then $m \in \ker d^i.$ To calculate $d^{r+1}(m)$, we must , $r$ times,  take the preimage by $i^*$. As \begin{equation}
	(i^*)^{r}(\overbrace{0,\cdots ,0}^r,\dhor(a_r),0,\cdots ,0))=(\dhor(a_r),0,\cdots ,0).
	\end{equation} 
	Therefore, projecting in the first coordinate yields the result.
\end{proof}
 We now prove the converse.
 \begin{thm}\label{Tliftctrit}
 	Let $m\in \E^{p,q}$. If $m\in~ ^rZ^{p,q}$, then $m$ is $r-$liftable.
 \end{thm}
	\begin{proof}
		For $r=1$ the proof is easy: $d^1(m)=\dhor(m) = 0 \in ~^1E^{p+1,q} = H^q(M^{p+1.\bullet})$ implies $d_h(m) = \dvert(a)$ for some $a\in M^{p+1,q-1}.$ 
		
		To illustrate the proof, we do the case $r = 2$. Let $m\in ~~^2Z^{p,q}$.  In particular,  $m\in~~ ^1Z^{p,q}$
		 which  implies that $m$ is $1-$liftable. Let $(m,a)$ be a $1$ -lift sequence. As $d^2(m) \in~^1B^{p+2,q-1}$, there is $a'$ with $\dvert(a') = 0$ such that $\dhor(a) = \dhor(a')$ in $^1E$. By the case $r=1$, there is $a''\in E^{p+2,q-2}$ with $\dvert(a'')=\dhor(a-a')$. Now, $(m,a-a',a'')$ is a $2$-lift sequence for $m$.
		
		For the general case, let $m\in~~ ^rZ^{p,q}$. Then, by induction, $m$ is $(r-1)-$liftable.  Let $(m,a_1,\cdots ,a_{r-1})$ be an $(r-1)$-lift sequence. Then by Proposition \ref{Tkerd},  $\dhor(a_{r-1})=d^r(m)= 0$ in $^rE$. Therefore, there is $a_{r-1}^{(1)}\in~~ ^{r-1}Z^{p,q}$ such that $\dvert(a_{r-1}) = 0$ and $\dhor(a_{r-1}) = d^{r-1}(a_1^{(1)})$ in $^{r-1}E$. Again by induction $a_1^{(1)}$ is $(r-2)$- liftable. Take an $(r-2)$-lift sequence   
	     \begin{equation}
		(0,a_1^{(1)},a_2^{(1)},\cdots ,a_{r-1}^{(1)})
		\end{equation}
		such that $\dhor(a_{r-1}-a^{(1)}_{r-1}) = 0$  in $^{r-1}E$. It is immediate to see that 
		\begin{equation} 
		(m,a_1-a_1^{(1)},a_2-a_2^{(1)},\cdots ,a_{r-1}-a_{r-1}^{(1)})
		\end{equation}
		is an $(r-1)-$lift sequence. Now we may construct inductively, for each $1\leq i\leq r-1$, $(r-1-i)$-lift sequences 
		\begin{equation}
		(0,\cdots ,0,a_i^{(i)},\cdots ,a_{r-1}^{(i)})
		\end{equation}
		such that $\dhor(a_{r-1}-a_1^{(i)}-\cdots -a_{r-1}^{(i)}) = 0$ in $^{r-i}E$. If $i = r-1$, we have that $\dhor(a_{r-1}-\sum_{i=1}^{r-1}a_{r-1}^{(i)}) = 0$ in $^1E$, that is, there is $a_r$ such that \begin{equation}
		\dhor(a_{r-1}-\sum\limits_{i=1}^{r-1}a_{r-1}^{(i)}) = \dvert(a_r).
		\end{equation}  
		It follows that \begin{equation}
		(m,a_1-a_1^{(1)},a_2-\sum\limits_{i=1}^{2}a_{2}^{(i)},\cdots ,a_{r-1}-\sum\limits_{i=1}^{r-1}a_{r-1}^{(i)},a_r)
		\end{equation}
		is an $r$-lift sequence.
	\end{proof}
The same method of the proof of Theorem \ref{Tkerd} gives us the following characterization of $^rB^{p,q}$.
\begin{thm}\label{Timd}
	Let $m\in \E^{p,q}$. Then $m\in~ ^rB^{p,q}$ if and only if there is an $(r-1)$-lift sequence $(a_1,\cdots ,a_{r-1},m)$ and an element $a_r$ such that $m = \dhor(a_{r-1})+\dvert(a_r)$.
\end{thm} 
	Again, considering the vertical filtration of the first quadrant double complex $\double$, the inclusions \begin{equation}
	i: \Tot(\double)^p\rightarrow\Tot(\double)
	\end{equation}
	induce a filtration 
	\begin{equation}\label{Efilt}
	H^k(\Tot(\double))\supset i^*(H^k(\Tot(\double)^1))\supset i^*(H^k(\Tot(\double)^2))\supset\cdots \supset i^*(H^k(\Tot(\double)^{p}))\supset\cdots 
	\end{equation}
	
	The following theorem explicitly explains the convergence of spectral sequences.
	\begin{prop}\label{Tconvergence}
		Let $E^{\bullet,\bullet}$ be a first quadrant bi-complex with the vertical filtration. Then the map
		 $$\begin{array}{ccc}
		\varphi_{p,q}:i^*(H^k(\Tot(M^{\bullet,\bullet})^p))/i^*(H^k((\Tot(M^{\bullet,\bullet})^{p+1})~)&\rightarrow&^\infty E^{p,q}\\
		(m,a_{p+1},\cdots ,a_k)&\mapsto &\bar{m} 
		\end{array},$$
		($q = k-p$) is a well defined  isomorphism.
	\end{prop}  
\begin{proof} The proof is an straightforward application of Theorems  \ref{Tliftctrit} and \ref{Timd}.
\end{proof}

To finish this section, we give some remarks: we have chosen to work with a first-quadrant double complex with differentials increasing the degree for the sake of clarity. All the definitions and arguments can be adapted to double complexes placed in other quadrants or with other configurations for the differentials. For the convergence, the proof of Theorem \ref{Tconvergence} shows that the only thing we need is a complex with finite diagonals.

\section{Buchsbaum-Eisenbud complexes are strands of  Koszul-\v{C}ech spectral sequences}

	To understand the structure of the disguised residual intersection, we need an explicit expression for a transgression map in a Koszul-\v{C}ech spectral sequence.  This expression 
	is exactly the same as the connecting map in the Buchsbaum-Eisenbud complexes explained in \cite[A2.6]{Eis}.  The surprising fact is that the whole family of Buchsbaum-Eisenbud complexes can be explained by the complexes Koszul and \v{C}ech. This, we will explain in this section.
	\subsection{Preliminaries}
	 We explain the new structures as much complete as possible here and refer to  \cite[A2]{Eis} for the details about the  family of complexes $\mathcal{C}^i_{\bullet}$; Where  $\mathcal{C}^0_{\bullet}$ is the Eagon-Northcott complex and $\mathcal{C}^1_{\bullet}$ is the Buchsbaum-Rim complex.
	 Besides some basic notations, we will recall some notations and the basic structure of the Buchsbaum-Eisenbud family of complexes. We avoid repeating the whole well-known structures. Instead, we keep the same notations and names as in  \cite[A2.6]{Eis}, for the sake of consistency, and refer the reader to loc. cit.  
	
	Let $R$  be a commutative  ring. For an   $R$-module $M$, we denote by $M^*$ the dual $\Hom_R(M,R)$.
	
	Recall that for $R$-module $G$ and any linear map $\varphi\in G^*$ one can define a differential $\partial_\varphi: \bigwedge G\rightarrow \bigwedge G$, called the Koszul differential of the linear map $\varphi$; see [BH, Definition 1.6.1]. This defines an action of $G^*$ on $\bigwedge G$ given by $$\psi\cdot w = \partial_\psi(w).$$ Clearly, if $\varphi,\psi\in G^*$, then $\varphi\cdot\psi\cdot w = -\psi\cdot\varphi\cdot w$ and this gives a natural action of $\bigwedge G^*$ on $\bigwedge G$
.	
    
    Let $\Phi: F=R^f\to G = R^g$ be a linear map. Then we can construct the generalized Koszul complex $\mathbb K(\Phi)$, defined in \cite{V} as follows: Let $S(G)$ be the symmetric algebra of $G$ and $\phi$ be the composition of the maps $$F\otimes_R S(G)\xrightarrow{\Phi\otimes 1} G\otimes_R S(G) \rightarrow S(G),$$ where the alst map is just the multiplication on $S(G)$. Then we can construct the Koszul differential $$\partial_\phi: \bigwedge F\otimes_R S(G) \rightarrow \bigwedge F\otimes_R S(G).$$
    The module $\bigwedge G^*$ acts on $\bigwedge F$ as follows. For any   $\eta\in G^*$ and any $v\in \bigwedge F$:
	
	\begin{equation}\label{Eactionf}
	\eta\cdot v := \Phi^*(\eta)\cdot v = (\eta\circ\Phi)\cdot v.
	\end{equation}
	Moreover, the module $G$ acts naturally on $S(G)$ via the multiplication. Therefore, we have an action of the module $G^*\otimes_R G$ on $\bigwedge F\otimes_R S(G)$ given by \begin{equation}\label{Eactions}
	(\eta\otimes u)\cdot(v\otimes w) = \eta\cdot v \otimes uw.
    \end{equation}
    
    The differential of the complex $\mathbb K(\Phi)$ can be explained via this action: it is just the action of the element $c$ which is the pullback of $1$ via the natural evaluation map $G\otimes G^* \rightarrow R$. Moreover for any  $S(G)$-module $M$,  one can define the generalized Koszul complex with coefficients in $M$ given by $\mathbb K(\Phi)\otimes_{S(G)}M$. Again, $G^*\otimes G$ acts on $\bigwedge F \otimes M$ and the differential of $\mathbb K(\Phi)\otimes_{S(G)}M$ is again given by the action of the element $c$.
    
    \begin{defi}\label{Pepsilon} A connecting map of degree $d$  for the map $\Phi$ is a map of the form $$\varepsilon_d: \bigwedge^{d+g} F\rightarrow \bigwedge^d F$$ given by the action of a generator $\gamma$ of $\bigwedge^{g}G^*.$
    \end{defi}

    \begin{defi}\label{Deagon}
	Let $\Phi: F=R^f\to G = R^g$ be a linear map with $f\geq g$. 	The  complex obtained by joining $(K(\Phi)_{(f-g-d)})^*$ and $K(\Phi)_{(d)}$ via a connecting map map $\varepsilon_{d}$ is denoted by $\mathcal{C}^d(\Phi)$. $\{\mathcal{C}^d(\Phi)\}_{d \in \mathbb{Z}}$ is the family of Buchsbaum-Eisenbud complexes.  The complex $\mathcal C^0(\Phi)$ is called the Eagon-Northcott complex, and the complex $\mathcal C^1(\Phi)$ is the Buchsbaum-Rim complex.
	\end{defi}
	Let $T_1,\dots,T_g$ be a basis of $G$. Then $S(G)\simeq R[T_1,\cdots,T_g]$ is a polynomial extension and has  the standard  grading, and the complex $\mathbb K(\Phi)$ is graded as well. 
    
	In terms of the basis $T_1,\dots,T_g$ the map $\Phi$ has a decomposition $$\Phi(w) = \phi_1(w)T_1+\cdots \phi_g(w)T_g,$$ where $\phi_i \in F^*$ for all $i$. In this case the differential $\partial_\phi$ of $\mathbb K(\Phi)$ can be expressed as $$\partial_\phi(w) = \partial_{\phi_1}(w)T_1+\cdots+\partial_{\phi_g}(w)T_g.$$
	
	The connecting map $\varepsilon_d$ can also be easily described in terms of the basis $T_1,\dots,T_g$.
	

	Let $T_1',\dots,T_g'$ be the basis of $G^*$ which is dual to $T_1,\dots,T_g$. Then the associated connecting map, $\varepsilon_d$, associated to the generator $T_1'\wedge\cdots\wedge T_g'$ is given by $$w \rightarrow \partial_{\phi_1}\cdots\partial_{\phi_g}(w).$$ 
	
	\subsection{The New Structure}\label{Snewstructure} Let $R$ be a commutative ring, $\Phi: F= R^f \to G=R^g, f\geq g$ a linear map and $K_\bullet=\mathbb K(\Phi)$ it's generalized Koszul complex. Let $T_1,...,T_g$ be a basis for $G$ and $\Sym(G)=S=R[T_1,\cdots,T_g]$. Moreover, let  $\check C^\bullet_{\ft}$ be the \v{C}ech complex of $S$ with respect to the sequence $\mathfrak t=(T_1,\cdots ,T_g)$. Consider then the bi-complex $K_\bullet\otimes_S\check C^\bullet_{\ft}$ and write $\mathcal{D}_\bullet = \Tot(K_\bullet\otimes_S\check C^\bullet_{\ft})$.  We display $K_\bullet\otimes_S\check C^\bullet_{\ft}$ as a third-quadrant bi-complex with $K_0\otimes_R C^0_\mathfrak t(S)$  at the origin of the plane.

     We begin to explain  our construction by looking at the spectral sequence coming from the vertical filtration. As $H^i_{\mathfrak t}(S) = 0$ for $i<g$, the spectral sequence collapses on the $g$-th row on the complex 
		
	$$	
	^1E^{-p,-q}_{ver} =\begin{cases}
	0, \text{~if~} q\neq g\\
	H^g_{\mathfrak t}(K_p),\text{otherwise,}
	\end{cases}
	$$
	\begin{equation}
 \xymatrix{
 		^1E^{*,-g}_{ver}:~~ 0\ar[r]& H^g_{\mathfrak t}(K_f)\ar[r]& \cdots \ar[r]& H^g_{\mathfrak t}(K_1)\ar[r]& H^g_{\mathfrak t}(K_0)\ar[r]& 0.\\
	}
	\end{equation}
	
	Since $K_i= S^{\binom{f}{i}}(-i)$ and $H^g_\mathfrak t(S)_{(d)} = 0$ for $d > -g$, we have, for a fixed degree $d$:
	\begin{equation}\label{Ekoszulcechdegree}
	(^1E_{vert})_{(d)}^{*,-g}: 0\rightarrow H^g_{\mathfrak t}(K_f)_{(d)}\rightarrow H^g_{\mathfrak t}(K_{f-1})_{(d)}\rightarrow \dots\rightarrow H^g_{\mathfrak t}(K_{g+d+1})_{(d)}\xrightarrow{\psi_d}H^g_{\mathfrak t}(K_{g+d})_{(d)}\rightarrow 0
	\end{equation}
	
	Then  $$(^2E^{-d-g,-g}_{ver})_{(d)} =(^\infty \E^{-d-g,-g}_{ver})_{(d)} \simeq \coker (\psi_d).$$
	Now, we look at the spectral sequence coming from the horizontal filtration. The second page of this spectral sequence is given by
	$$(^2E^{-p,-q}_{hor})_{(d)}=H^q_{\ft}(H_p(K_{\bullet}))_{(d)}. $$
	
 The $0-$th row of $~^0\E^{*,*}_{hor}$  in degree $d$ is the complex 
 \begin{equation}
	0\rightarrow (K_d)_{(d)} \xrightarrow{\mu_d}\cdots \rightarrow (K_0)_{(d)}\rightarrow 0;
\end{equation}

	so that $$(^2\E_{hor}^{-d,0})_{(d)}\simeq H^0_\mathfrak t(\Ker(\mu_d))\subseteq \Ker(\mu_d).$$ 
	
	By the convergence of the spectral sequences, we have 
	\begin{equation}
	(^\infty \E^{-d-g,-g}_{ver})_{(d)} \simeq H^d(\mathcal D_\bullet)_{(d)}.
	\end{equation} 
	Moreover, there is a filtration
	 $$H^d(\mathcal D_\bullet)_{(d)}=\mathcal F_{d,0}\supseteq\mathcal F_{d,1}\supseteq\cdots $$ such that 
	$$\frac{\mathcal F_{d,i}}{\mathcal F_{d,i+1}}\simeq ~^\infty (\E_{hor}^{-d-i,-i})_{(d)}.$$
	
	We then have a natural surjection $H^d(\mathcal D_\bullet)_{(d)} \to (^\infty E_{hor}^{d,0})_{(d)}$.  Define the map $\tau_d:H^g_\mathfrak t(K_{d+g})_{(d)}\rightarrow (K_d)_{(d)}$ to be the composition
	
	\begin{equation}\label{Dtd}
	H^g_\mathfrak t(K_{d+g})_{(d)} \rightarrow \coker(\psi_d) \xrightarrow{\sim} H^d(\mathcal D_\bullet)_{(d)}\rightarrow(~^\infty E^{-d,0}_{hor})_{(d)} \hookrightarrow (~^1E^{-d,0}_{hor})_{(d)} = \Ker(\mu_d)\hookrightarrow(K_d)_{(d)}.
	\end{equation}
	
	We define the complex $\mathcal K_d(\Phi)$ to be the  complex 
	\begin{equation}\label{Ekccomplex}
	0 \rightarrow H^g_\mathfrak t(K_f)_{(d)} \rightarrow\dots\rightarrow H^g_\mathfrak t(K_{d+g})_{(d)}\xrightarrow{\tau_d}(K_d)_{(d)}\rightarrow\dots\rightarrow (K_0)_{(d)}\rightarrow 0   
	\end{equation}
	
	The principal theorem of this chapter is the following:
	\begin{thm}\label{Tkcen}
		Let $\Phi: F = R^f \rightarrow G = R^g$ be a linear map  with $f\geq g$. Then the complexes $\mathcal{C}^d(\Phi)$ and $\mathcal K_d(\Phi)$ are isomorphic for $d\leq f-g$.
	\end{thm}
	 Fix $d\leq f-g$. The complexes $\mathcal C^d(\Phi)$ and $\mathcal K_d(\Phi)$ are isomorphic at the right side of the joining maps $\tau_d$ and  $\varepsilon_{d}$ as they are both the generalized Koszul complex of $\Phi$. Therefore we must study the left parts and the joining map of both complexes. 
	
	\begin{prop}\label{Prank}
		The left parts of $\mathcal C^d(\Phi)$ and $\mathcal K_d(\Phi)$ are isomorphic.
	\end{prop}
	\begin{proof}
	Recall that $H^g_\mathfrak t(S)$ has an inverse polynomial structure, and $H^g_\mathfrak t(S)_{(-g)}$ is a free $R$-module generated by the monomials $\frac{1}{T_1^{\alpha_1}\dots T_g^{\alpha_g}}$ with $\sum_{i=1}^g \alpha_i = g, \alpha_i \geq 1$.
	Thus there is  a perfect pairing $$S_{(d)}\otimes_RH^g_\mathfrak t(S)_{(-d-g)}\longrightarrow H^g_\mathfrak t(S)_{(-g)} \simeq R$$ given by multiplication. The isomorphism $(S_{(d)})^* \simeq H^g_\mathfrak t(S)_{(-d-g)}$ induced by this pairing takes the element $(T_1^{\alpha_1}\dots T_g^{\alpha_g})'$ to the element $\frac{1}{T_1^{\alpha_1+1}\dots T_g^{\alpha_g+1}}$.
	Using this paring, we have the following isomorphisms:
	$$\mathcal K_d(\Phi)_{d+i} = H^g_\mathfrak t(K_{g+d+i-1})_{(d)}\simeq$$ 
	$$H^g_\mathfrak t(\bigwedge^{d+g+i-1}F \otimes_R S(-g-d-i+1))_{(d)}\simeq$$
	$$\bigwedge^{g+d+i-1}F\otimes_RH^g_\mathfrak t(S(-g-d-i+1))_{(d)}\simeq$$ $$\bigwedge^{g+d+i-1}F\otimes_RH^g_\mathfrak t(S)_{(-g-i+1)} \simeq$$ $$\bigwedge^{g+d+i-1}F\otimes_R (S_{(i-1)})^* \simeq (\mathcal{C}^d)_{d+i},$$
	hence components of the complexes are isomorphic.\\
	For the differentials, notice that the left part of $\mathcal C^d$ is a strand of the generalized Koszul complex of $\Phi$ with coefficients in $S^*$ and the left part of $\mathcal K_d(\Phi)$ is a strand of the generalized Koszul complex of $\Phi$ with coefficients in $H^g_\mathfrak t(S)$. Both differentials are induced by the action of the element $c \in G^*\otimes G$ defined earlier in this section, and this action commutes with the isomorphism induced by the above perfect pairing. The proposition is therefore proved. 
	\end{proof}

It remains to analyze the joining maps $\tau_d: H^g_\mathfrak t(K_{g+d})_{(d)}\rightarrow (K_d)_{(d)}$ defined in (\ref{Dtd}) and $\varepsilon_{d}$ defined in Definition \ref{Pepsilon}.

We use the next notations in the sequel. 
	\begin{nota} \label{Ndenominators}
	\begin{itemize}
	    \item {	For any $L\subset \{1,\cdots ,g\}$ with $|L| = \ell$,  we define $\sgn(L)$ to be the sign of the permutation that put the elements of $L$ on the first $\ell$ positions.}
	\item{	For the set  of variables $T_1,\cdots ,T_g$  and $L\subset\{1,\cdots ,g\}$. We define $T_L := \prod\limits_{j\notin L}T_j$.}
		\item{	For a set of maps $\{\varphi_{\ell_i}: i=1,\cdots, m\}$ in $F^*$, we use the notation $\partial_L$ to denote the composition $\partial_{\varphi_{\ell_1}}\cdots\partial_{\varphi_{\ell_m}}$ where $L=\{\ell_1,\cdots, \ell_m\}$. }
		\end{itemize}
	\end{nota}
	
	Using the theorems developed in section \ref{SSS} about the structure of the maps in spectral sequences, we have an explicit description for $\tau_d$.	
	\begin{thm}\label{Tlifting}
		Let  $R$ be a commutative  
		ring, $\Phi: F = R^f \to G = R^g$ a linear map  with $f\geq g$ and $K_\bullet=\mathbb K(\Phi)$ the generalized Koszul Complex. 
		
		Let $T_1,\dots,T_g$ be a basis of $G$,  $\Phi=\sum_{i=1}^r \phi_i\cdot T_i$    and $S:=S(G)=R[T_1,\cdots ,T_g]$.
		
		Consider the double complex 
		$K_\bullet\otimes_S\check{C}^\bullet_\mathfrak t(S)$ and its horizontal and vertical spectral sequences. Then, for each $0\leq d \leq f-g$ and $w\in K_{g+d}$, the element \begin{equation}\label{eqm}
		   	m = w\otimes \frac{1}{T_1\dots T_g}\in K_{g+d}\otimes\check{C}^g_{\mathfrak{t}}(S) 
		\end{equation} 
		is $g$-liftable. Moreover, the $i$-th lift of this element is, up to a sign,  
		$$\sum\limits_{|L|=i}\partial_L(w)\otimes \sgn(L)\frac{1}{T_L} .$$ 
		In particular, 
		$$\tau_d(m) = \partial_{\phi_1}\cdots\partial_{\phi_g}(m).$$
		 So that $\tau_d=\varepsilon_d$.
	\end{thm}
	\begin{proof} 
	
	In the course of the proof  we do all the liftings without concerning  the  signs coming from the bi-complex, for the sake of clarity. Finally, we stress that the involved signs  depends only on $g$.
	
				In $(^0\E_{ver}^{-d-g,-g})_{(d)}$ the differential $d^0=\dvert$ is zero. By (\ref{Ekoszulcechdegree}), $(^1\E_{ver}^{-d-g+1,-g})_{(d)}=0$ and $(^r\E_{ver}^{-d-g+r,-g-r+1})_{(d)}=0$ for all $r\geq 2$. Therefore all differentials $d^r$ are zero in $(^r\E_{ver}^{-d-g,-g})_{(d)}$ and $(^0\E_{ver}^{-d-g,-g})_{(d)}= (^\infty\Z_{ver}^{-d-g,-g})_{(d)}$. Then by Theorem \ref{Tliftctrit}, $m$ is $g$-liftable. 
				
				We now show the formula for the $i$th lift by using an induction on $i$. The process  of lifting is by applying the horizontal map, Koszul differentials,  in each step and then taking the pre-image by the vertical maps which are the \v{C}ech differentials. 
				\begin{equation}\label{Edif1}
		d_h(m)=\sum\limits_{i=1}^g\partial_i(w)\otimes \frac{T_i}{T_1\dots T_g}
		\end{equation}

		It's immediate to see that (\ref{Edif1}) is a \v{C}ech boundary, and the 1-lift is given by \begin{equation}
		\partial_1(w)\otimes\frac{1}{T_{\{1\}}} - \partial_2(w)\otimes\frac{1}{T_{\{2\}}}+\cdots+(-1)^g \partial_g(w)\otimes\frac{1}{T_{\{g\}}}, 
		\end{equation}
		
		as required. 
		
		Now suppose that $i\geq 1$ and that  the  $i$-th lift of $m$ is given by 
		\begin{equation}\label{Ecoef3}m_i=\sum\limits_{|L|=i}\partial_L(w)\otimes \sgn(L)\frac{1}{T_L}.
		\end{equation}

		Applying the Koszul differential, we then have 
		
		\begin{equation}\label{Ecoef4}
		    d_h(m_i)=\sum\limits_{|L|=i}(\sum\limits_{i=1}^g \partial_i\partial_L(m)T_i) \otimes \sgn(L)\frac{1}{T_L}.
		\end{equation}
			Notice that if $i\in L$, then $\partial_i\partial_L = 0$.  This shows that $d_h(m_i)$ is the image of  the \v{C}ech's map of the element
			\begin{equation}\label{Ecoef5}m_{i+1}=\sum\limits_{|L|=i+1}\partial_L(w)\otimes \sgn(L)\frac{1}{T_L}.
		\end{equation}
	as desired. 
	
		To see that $\tau_d = \varepsilon_d$, we analyze the construction of $\tau_d$ given in (\ref{Dtd}). Since $$ H^g_\mathfrak t(\bigwedge\limits^{d+g}(S(-d-g))^f)\simeq\bigwedge\limits^{d+g}F\otimes H^g_\mathfrak t(S)_{(-g)},$$ any element of $\coker(\varphi_d)$  is presented by  an element
	$$	m = w\otimes \frac{1}{T_1\dots T_g}\in K_{g+d}\otimes\check{C}^g_{\mathfrak{t}}(S). $$

		By theorem \ref{Tconvergence}, the isomorphism $\coker(\varphi_d)\simeq H^d(\mathcal D_\bullet)_{(d)}$ sends $m$ to the cohomology class $(m,-a_1,...,(-1)^ga_g)\in H^d(\mathcal D_\bullet)_{(d)}$ where $(m,a_1,\cdots,a_g)$ is a $g$-lift sequence for $m$.
		
		Again by Theorem \ref{Tconvergence}, the map $H^d(\mathcal D_\bullet)_{(d)}\rightarrow ^\infty\E^{-d,0}$ sends the cohomology class $(m,-a_1,...,(-1)^ga_g)$ to the element $(-1)^ga_g$. As $a_g$ is just the $g$-lift of $m$ calculated above. Comparing with Definition \ref{Pepsilon}, we see that $\tau_d = \varepsilon_d$ up to a sign.
		\end{proof}
	The proof of Theorem \ref{Tkcen} now follows from Proposition \ref{Prank} and Theorem \ref{Tlifting}.
 $\mathcal C^d_\bullet(\Phi)$.	

	\section{Generators of Residual Intersections}
	In this section, we  recall the definition of  ``Residual Approximation Complexes" and ``Disguised Residual Intersections". Based on the materials developed so far, we show how Kitt ideals (defined in the introduction) can approximate a general colon ideal $(\fa:I)$. By {\it general colon ideal} we mean,  in most of the coming results, $J=\fa:I$  needs not to be a residual intersection. However, it is not true that $\Kitt(\mathfrak a,I)=(\fa:I)$ for any colon ideal, see  Example \ref{Example}.
	
	\subsection{Definitions and known results}
	\begin{defi}\label{Dresidual}
		Let $R$ be a commutative ring and $I\subset R$ an ideal of grade $g$ and $s\geq g$. We say that an ideal $J$ is an algebraic $s$-residual intersection of $I$, if $J = (\mathfrak{a}:I)$, with $\mathfrak{a}=(a_1,\dots,a_s)\subset I$, and ht($J$)$\geq s $. Moreover, we say that \begin{enumerate}
			\item $J$ is an arithmetic $s$-residual intersection if $\mu_{R_\mathfrak{p}}(I/\mathfrak{a})_\mathfrak{p}\leq 1$ for all prime ideals $\mathfrak{p}$ with $\Ht(\mathfrak{p}) = s$
			\item $J$ is a geometric $s$-residual intersection if ht$(I+J)\geq s+1$
		\end{enumerate} 
	\end{defi}
	
	In  \cite{Ha}, \cite{HN}, \cite{CNT} the authors tackle the problem of the Cohen-Macaulayness of an $s$-residual intersection $J = (\fa:I)$ without assuming that $I$ satisfies the $G_s$-condition\footnote {An ideal $I$ satisfies the $G_s$-condition if $\mu(I_{\fp})\leq \Ht(\fp)$ for any prime ideal $ \fp \in {\rm V}(I)$ of height at most $s-1$.}. These works rely upon the construction of a family of complexes, called the {\it Residual Approximation Complexes}, that we now describe.
	 
	Let $R$ be a commutative  ring, $I = (f_1,\cdots ,f_r)$ an ideal of R and $\mathfrak{a} = (a_1,\cdots ,a_s)\subseteq I$. For each $1\leq j\leq s$, let  $a_j=\sum\limits_{i=1}^rc_{ij}f_i$ for some $c_{ij}\in R$. Let $S = R[T_1,\dots,T_r]$ be a standard  polynomial extension of $R$ and $\gamma_j=\sum\limits_{i=1}^rc_{ij}T_i\in S_1$. We then consider the $\mathcal{Z}$-complex, $\mathcal{Z}_\bullet(\ff;R)$, of Herzog-Simis-Vasconcelos \cite{HSV} 
	\begin{equation}
	0\rightarrow Z_r(\ff;R)\otimes_RS(-r)\rightarrow Z_{r-1}(\ff;R)\otimes_RS(-r+1)\rightarrow\dots\rightarrow Z_0(\ff;R)\otimes_RS\rightarrow 0
	\end{equation}
	
	where $Z_i(\ff;R)$ stands for the $i$-th module of Koszul cycles of the sequence $\ff=(f_1,\dots,f_r)$. 
	
	In the sequel we consider the bold form
	$\ff, \aa$ and $\ag$ for the sequences $f_1,\cdots ,f_r$, $a_1,\cdots ,a_s$ and $\gamma_1,\cdots, \gamma_s$.
	
	Now, consider a new complex, given by $$\mathcal D_\bullet =\Tot(\mathcal Z_\bullet(\ff;R)\otimes_SK_\bullet(\ag,S))$$
	where $K_{\bullet}(\ag,S)$ is the Koszul complex $K_\bullet(\gamma_1,\dots,\gamma_s,S)$. 
	The $i$-th component of this complex is given by 
	\begin{equation}
	\mathcal D_i = \bigoplus\limits_{k=i-s}^{\min\{i,r\}}(Z_k(\ff;R)\otimes_RS(-k))\otimes_S S^{\binom{s}{i-k}}(-i+k) \simeq \bigoplus\limits_{k=i-s}^{\min\{i,r\}}(Z_k(\ff;R)\otimes_RS^{\binom{s}{i-k}})(-i) 
	\end{equation}
	We then tensor  $\mathcal D_\bullet$ to the \v{C}ech complex $\check C^\bullet_\mathfrak t(S)$, where $\mathfrak{t}=(T_1,\dots,T_r),$ and repeat the  the same procedure as in Section \ref{Snewstructure} to glue  the horizontal spectral sequence to the vertical spectral sequence. Thence  for each degree $k$, we have the complex 
	\begin{equation}
	_k\mathcal{Z}^+_\bullet : 0\rightarrow H^r_\mathfrak t(\mathcal D_{r+s-1})_{(k)}\rightarrow\dots\rightarrow H^r_\mathfrak t (\mathcal{D}_{r+k})_{(k)} \xrightarrow{\tau_k}(\mathcal D_{k})_{(k)}\rightarrow\dots\rightarrow(\mathcal D_0)_{(k)}\rightarrow0.
	\end{equation}
	\begin{defi}
		The complex $_k\mathcal Z^+_\bullet$ constructed above is called the $k$-th {\it residual approximation complex} with respect to the generating sets $\ff$ and $\aa$ of $I$ and $\fa$.
	\end{defi}
	
		
	
	Having these approximation complexes in hand, the first interesting  property is their  acyclicity.  Here is where  the Cohen-Macaulayness of the ring $R$ and  sliding depth conditions come into play.  We recall the definitions here. 
	\begin{defi}\label{dsd} Let $(R,\fm)$ be a Noetherian local ring of dimension $d$ and $I=(f_1,...,f_r)=(\ff)$ an ideal. Let $k$  be an integer. We say that the
ideal $I$   satisfies $ \SD_k$ if $\depth(H_i(\ff;R))\geq \min\{d-g,d-r+i+k\}$ for all $i\geq 0$ ; also $\SD$ stands for
 $\SD_0$. Similarly, we say that $I$ satisfies the
sliding depth condition on cycles, $\SDC_k$, at level $t$, if
$\depth(Z_i(\ff;R))\geq \min\{d-r+i+k, d-g+2, d\}$ for all $r-g-t\leq i\leq r-g$. 

$I$ is strongly Cohen-Macaulay, SCM, if $H_i(\ff;R)$ is CM for all $i$. 
\end{defi}

	We state an acyclicity criterion.
	\begin{thm}\cite[Theorem 2.6]{HN}
		Let $(R,\fm)$ be a  CM local ring of dimension $d$, $I=(f_1,...,f_r)$ an ideal with height $g \geq 1$. Let $s\geq g$ and fix $0\leq k\leq s-g+2$. Suppose that $I$ satisfies $\SD$ and that one the following hypotheses holds
\begin{itemize}
\item[(i)] $r+k\leq s$ or, 
\item[(ii)] $r+k\geq s+1$ and $\depth(Z_i(\ff))\geq \min\{d,d-s+k\}$ for $0\leq i\leq k$.
\end{itemize}
Then for any $s$-residual intersection $J=(\fa:I)$, the complex $_k\mathcal{Z}^{+}_{\bullet}$ is acyclic. Furthermore, $H_0(_k\mathcal Z^+_{\bullet})$ is a Cohen-Macaulay $R$-module of dimension $d-s$.
	\end{thm}
	The last map of the complex $_0\mathcal Z^+_{\bullet}$ is 
	\begin{equation}\label{tau0}
	\tau_0:H^r_\mathfrak t(\mathcal{D}_{r})_0\rightarrow(D_0)_0 \simeq R; 
	\end{equation} 
	Hence $H_0(_0\mathcal{Z}^+_{\bullet}) \simeq R/K$ for some ideal  $K\subset R$ ideal. This ideal is called the {\it  disguised residual intersection}.
More specifically, 	

\begin{defi}\label{Ddisguised}
		Let $R$ be a commutative ring, $I=(f_1,\cdots ,f_r)$  and $\mathfrak a=(a_1,\cdots ,a_s)\subseteq I$  be ideals of $R$ and   $\Phi = (c_{ij})$  be an $r\times s$ matrix such that  $(\aa) = (\ff).\Phi.$
		
		Then  the {\bf  disguised residual intersection} of $I$ with respect to the representation matrix $\Phi$ is the image of the map 
		$$\tau_0:H^r_\mathfrak t(\mathcal D_r)_0 \rightarrow R$$
		 in $_0\mathcal Z_{\bullet}^+$-complex;  we denote it by $K(\aa,\ff,\Phi)$ . 
		
	\end{defi}

	There are some tight relations between the disguised residual intersection and the residual intersection as the following theorems state.
	\begin{thm}\label{TaboutK} Let $R$ be a commutative ring and keep the notations in the Definition \ref{Ddisguised}. Let $J=\fa:_R I$ and $K=K(\aa,\ff,\Phi)$ then 
		\begin{enumerate}
			\item $K\subseteq J$, and $V(K)=V(J)$, \cite[Theorem 2.11]{Ha}.
			\item $J = K$ if $\ff=(a_1,\cdots ,a_s, b)$, \cite[Theorem  4.4]{HN}.
		       \item If $R$ is Cohen-Macaulay local ring,  $I$ is an ideal of height $g\geq 1$ satisfying SDC$_1$ at level min$\{s-g,r-g\}$ and $J $ is an arithmetic  $s$-residual intersection then  $K = J$ and it is Cohen-Macaulay of height $s$,  \cite[Theorem 2.11]{Ha}.
		        \item If $R$ is Cohen-Macaulay local ring,  $I$ is an ideal of height $2$ satisfying SD$_1$  and $J$ is any algebraic   $s$-residual intersection then  $K = J$,   \cite[Theorem 4.5]{CNT}.
		 \end{enumerate}
\end{thm}
	
	In \cite[Theorem 4.5]{CNT} Chardin-Naeliton-Tran also shows that   if $R$ is Cohen-Macaulay local ring,  $I$ is an ideal of height $g\geq1$ satisfying SD$_1$ then any algebraic $s$-residual  intersection of $I$ is Cohen-Macaulay of height $s$. This  provides  an affirmative answer to the question of Huneke-Ulrich \cite{HU}.
	
	Hassanzadeh-Naeliton proposed the following conjecture in \cite{HN}.
	\begin{conj}\label{conj}
		If $R$ is a  Cohen-Macaulay local ring and $I$ an ideal satisfying SD and $depth(R/I) \geq d-s$. Then  any algebraic $s$-residual intersection of  $I$ coincides with  the disguised residual intersection.
	\end{conj}

	\subsection{The structure of  disguised residual intersections}
	An obstruction to generalizing the techniques in  \cite[Theorem 4.5]{CNT} for ideals $I$ with $\Ht(I) > 2$ is that the usual reduction modulo a regular sequence does not work smoothly if one does not know the explicit structure of the disguised residual intersection. This is why  in this paper we look for explicit descriptions of the maps in spectral sequences in order to determine the generators of disguised residual intersections.

	 Since the bi-complex ${\mathcal D}_{\bullet} = \Tot(K_{\bullet}(\ag;S)\otimes_S \mathcal Z_{\bullet} (\textbf{f};R) )$ is a subcomplex of the complex 
	${\mathcal F}_{\bullet} =K_{\bullet}(\ag,T_1,\cdots ,T_r;S)\simeq  \Tot(K_{\bullet} (\ag;S)\otimes_S K_{\bullet} (T_1,\dots T_r;S))$ the map $\tau_0$ in (\ref{tau0}) is just the restriction of the comparison map $\tau_0$ constructed in Section \ref{Snewstructure} for the bi-complex $\mathcal F_{\bullet}  \otimes_S \check C^\bullet_\mathfrak t$. By Theorem \ref{Tlifting}, this map is just the connecting map in the Eagon-Northcott complex \begin{equation}\label{MatrixM}
	\varepsilon_0:\bigwedge^r R^{r+s} \rightarrow R,
	\end{equation}   
	defined in Proposition \ref{Pepsilon},  associated to the matrix $M = (c_{ij}| id_{(r\times r)})$.

	 For the rest of this section, we set $e'_1,\dots,e'_s,e_1,\dots,e_r$  for the basis of $K_1(\gamma_1,\cdots,\gamma_s,T_1,\cdots,T_r;S)$ as a free $S$-module; so that
	   $\partial(e'_i) = \gamma_i$	and  $\partial(e_i) = T_i$, in the corresponding Koszul complex. 	We also denote the columns of the matrix $M$ with $C_1,\cdots, C_s,I_1,\cdots, I_r$. 
	
	
	\begin{lemma}\label{Peagonres}
		 Keeping the above notations,  let $L_1\subseteq\{1,\cdots ,s\}$and $L_2\subseteq\{1,\cdots ,r\}=L$ such that $|L_1|+|L_2|=r$. Then 
		 $$\varepsilon_0(e'_{L_1}\otimes e_{L_2})=\bigwedge_{i \in L_1} C_i\bigwedge_{j\in L_2}I_j$$
		 where $C_i$ and $I_j$ are the columns of $M$ considered as elements in $\bigwedge ^1 R^r$. 
		 In other words,  $$\varepsilon_0(e'_{L_1}\otimes e_{L_2})=\pm\det\Phi^{L_1}_{L\setminus L_2}.$$
where $\Phi^{L_1}_{L\setminus L_2}$ is the sub-matrix of $\Phi = (c_{ij})$ whose  rows indexed by $L\setminus L_2$ and columns indexed by $L_1$.
	\end{lemma}
	\begin{proof}
	The complex $\mathcal F_\bullet$ is just the generalized Koszul complex of the linear map $\Psi$ presented by the matrix $M$. According to the notations in Definition \ref{Deagon},  the differential of this complex  $\partial_\Psi$  can be expressed as 
	$$\partial_\Psi(w) = \partial_{\psi_1}(w)T_1+\cdots+\partial_{\psi_r}(w)T_r$$
	where $\psi_i$ is determined by the $i$-th row of the matrix $M$.
 Henceforth  $$\varepsilon_0(w)= \partial_{\psi_1}\cdots\partial_{\psi_r}(w).$$
	Therefore the equality $$\varepsilon_0(e'_{L_1}\otimes e_{L_2})=\bigwedge_{i \in L_1} C_i\bigwedge_{j\in L_2}I_j$$ follows from the elementary properties of the exterior product.
	
For the second expression, one has 
$$\bigwedge_{i \in L_1}C_i\bigwedge_{j\in L_2}I_j =\det(\Phi^{{L_1}}| id_{(r\times r)}^{L_2})=\pm\det\begin{bmatrix}
		\Phi^{L_1}_{L\setminus {L_2}} & 0\\ * & id_{|{L_2}|\times|{L_2}|}
		\end{bmatrix} = \pm\det\Phi^{{L_1}}_{L\setminus {L_2}}$$ 
	
	\end{proof}


	The following theorem explains the generators of the disguised residual intersection. 
		\begin{thm}\label{Tgensdisguised2}
		Let $R$ be a commutative ring, $I = (f_1,\cdots ,f_r)\subseteq R$, $\mathfrak a = (a_1,\cdots ,a_s)\subseteq I$ ideals and $\Phi =(c_{ij})$ a matrix such that $(\aa)=(\ff)\cdot \Phi$. Consider the  differential graded algebra $ K_\bullet(\ff;R)=R\langle e_1,\cdots,e_r;\partial(e_i)=f_i\rangle$. 
		Let $\zeta_j = \sum_{i=1}^r c_{ij}e_i,1\leq j\leq s$, $\Gamma_\bullet=R\langle\zeta_1,\cdots ,\zeta_s\rangle$ the sub-algebra generated by the $\zeta$'s, and $Z_\bullet=Z_\bullet(\ff;R)$ the sub-algebra of Koszul cycles. Then the disguised residual intersection $$K(\aa,\ff,\Phi)=<\Gamma_\bullet\cdot Z_\bullet>_r.$$
	\end{thm}

\begin{proof}
	By Definition \ref{Ddisguised}, $K(\aa,\ff,\Phi)$ is the image of the map  $$\tau_0:H^r_\mathfrak t(\mathcal D_r)_0 \rightarrow R.$$ 
	
	Since 
  $$\mathcal D_r = \bigoplus\limits_{k=r-s}^{r} S^{\binom{s}{r-k}}(-r+k)\otimes_S(Z_k(\ff;R)\otimes_RS(-k)) =\bigoplus\limits_{k=r-s}^{r}\bigwedge^{r-k}R^s\otimes_RZ_{k}(\ff,R)\otimes_RS(-r),
  $$
  $$H^r_\mathfrak t(\mathcal D_r )_0=\bigoplus\limits_{k=r-s}^{r}\bigwedge^{r-k}R^s\otimes_R Z_{k}(\ff,R)\otimes_RH^r_\mathfrak t(S)_{-r}.$$
  The map $\tau_0$ above is the restriction of the same named map in Theorem \ref{Tlifting} which is, up to a sign, equal to the connecting map in the  Eagon-Northcott complex,  $\varepsilon_0$.  
	So that it is enough to determine 
	$$\{\varepsilon_0(  e'_{L_1}\otimes z_{j}):r-s\leq j \leq r,  |L_1|= r-j,  \text{~~and~~} z_{j}\in Z_{j}(\ff;R) \}.$$
		More explicitly, for a cycle $z_{j} = \sum_{|L_2| = j}\alpha_{L_2}e_{L_2}$, according to the proof of Lemma	\ref{Peagonres}, 
		$$\varepsilon_0(e'_{L_1}\otimes z_{j})= \sum_{|L_2| = j}\alpha_{L_2}\partial_{\phi_1}\cdots\partial_{\phi_r}(e'_{L_1}\otimes e_{L_2})=\sum_{|L_2| = j}\alpha_{L_2}\bigwedge_{i \in L_1}C_i\bigwedge_{j\in L_2}I_j$$
	
	Now, identifying $C_i$ with $\zeta_i = \sum_{k=1}^r c_{ki}e_k\in \bigwedge^1R^r$ and $I_j$ with $e_j$, we get
	
	$$\varepsilon_0(e'_{L_1}\otimes z_{j})=\sum_{|L_2| = j}\alpha_{L_2}\bigwedge_{i \in L_1}C_i\bigwedge_{j\in L_2}I_j=\sum_{|L_2| = j}\alpha_{L_2}\bigwedge_{i \in L_1}\zeta_i\bigwedge_{j\in L_2}e_j=(\bigwedge_{i \in L_1}\zeta_i)\bigwedge z_j$$
	
	We also notice that $<\Gamma_\bullet\cdot Z_\bullet>_r\subseteq K_r(\ff;R) = \bigwedge^rR^r\simeq R.$ Hence $<\Gamma_\bullet\cdot Z_\bullet>_r$ is isomorphic to an ideal of $R$. 
\end{proof}
\begin{rmk} As it may be understood from the proof of the above theorem, to construct $K(\aa,\ff,\Phi)$ one may only need cycles of order $j\geq r-s$. However this fact is implicit in the equation $K(\aa,\ff,\Phi)=<\Gamma_\bullet\cdot Z_\bullet>_r$. Since $\Gamma_\bullet=R\langle\zeta_1,\cdots ,\zeta_s\rangle$, an element of $\Gamma_\bullet$ has degree at most $s$.  
\end{rmk}	
	
	We postpone the corollaries of this structural theorem until section 5. We first show that the definition of the disguised residual intersection does not depend on any choices of generators of $\mathfrak a$ and $I$ and neither on the choices of the matrix $\Phi$. 
	\subsection{Independence from the generating sets}\label{SSindependent}

	In this subsection let  $R$ be a commutative ring,  $I = (f_1,\cdots ,f_r)$, $\mathfrak a = (a_1,\cdots ,a_s)\subseteq I$ and $(\aa)=(\ff)\Phi$ for some matrix $\Phi=(c_{ij})$. 
	
	
	\begin{prop}\label{Pindmatrix} The desguised residual intersection does not depend on the presentation matrix. 
		    
	\end{prop}
	\begin{proof}
	Let $\Phi=(c_{ij})$ and $\widetilde{\Phi}=(\widetilde{c_{ij}})$ be two matrices such that $(\aa)=(\ff)\Phi =(\ff)\widetilde\Phi$.  
		As $(\ff)(\Phi-\widetilde{\Phi})=(a_1-a_1,\cdots ,a_s-a_s) = \textbf{0}$, the columns of the matrix $\Phi-\widetilde{\Phi}$ are syzygies of the sequence $\ff$. Hence setting $\zeta_j = \sum_{i=1}^r c_{ij}e_i$, $\widetilde{\zeta_j} = \sum_{i=1}^r\widetilde{c_{ij}}e_i$, $\Gamma_\bullet = R\langle\zeta_1,\cdots ,\zeta_s\rangle$, $\widetilde{\Gamma}_\bullet = R\langle\widetilde{\zeta_1},\cdots ,\widetilde{\zeta_s}\rangle$, and $Z_\bullet$ the algebra of Koszul cycles of the sequence $\ff$, we have, by Theorem \ref{Tgensdisguised2}, \begin{equation}
		K(\aa,\ff,\Phi) = <\Gamma_\bullet\cdot Z_\bullet>_r
		\end{equation}
		and \begin{equation}
		K(\aa,\ff,,\widetilde\Phi) = <\widetilde\Gamma_\bullet\cdot Z_\bullet>_r
		\end{equation}
		Since, for all $j$, $\zeta_j = \widetilde{\zeta_j} + z_j$ for some $z_j \in Z_1$, we have \begin{equation}
		\widetilde{\Gamma}_i \subseteq \Gamma_i + \Gamma_{i-1}\cdot Z_1+\dots+\Gamma_{1}\cdot Z_{i-1} + Z_i.
		\end{equation}
		Hence for $1\leq i \leq s$ \begin{equation}\label{Egammasyz}
		\widetilde{\Gamma}_i\cdot Z_{r-i} \subseteq \Gamma_i\cdot Z_{r-i}+\Gamma_{i-1}\cdot Z_{r-i+1}+\dots+\Gamma_{1}\cdot Z_{r-1} + Z_r
		\end{equation}
		This proves the inclusion $$K(\aa,\ff,\Phi)\supseteq K(\aa,\ff,\widetilde{\Phi})$$
		The opposite inclusion follows similarly.
	\end{proof}
	
	

	\begin{prop}\label{Pindgensa}
		The disguised residual intersection does not depend on the choice of generators of $\mathfrak a$.
	\end{prop}
	\begin{proof}
		Let $(a_1,\cdots ,a_s)=(\aa)$, $(a'_1,\cdots ,a'_{s'})=(\aa')$ be two generating sets of the ideal $\mathfrak a$. There exists an $s\times s'$ matrix $M$, and an $s'\times s$ matrix $M'$ such that \begin{equation}
		(\aa)\cdot M = (\aa')
		\end{equation}
		\begin{equation}
		(\aa')\cdot M' = (\aa)
		\end{equation}
		
		Therefore, choosing $\Phi$ such that $(\aa)=(\ff)\cdot \Phi$, we have
		\begin{equation}
		(\aa')=(\ff)\cdot\Phi\cdot M.
		\end{equation} 
		Let $\zeta_j, 1\leq j \leq s$ be the $\zeta$'s associated to the matrix $\Phi$ and $\zeta_j', 1\leq j\leq s'$ be the $\zeta$'s associated to the matrix $\Phi\cdot M$. By elementary properties of the wedge product, any wedge product of the $\zeta_j'$'s is a linear combination of wedge products of the $\zeta_j$'s with coefficients some minors of the matrix $M$. 
		Hence, by Theorem \ref{Tgensdisguised2}, \begin{equation}
		K(\aa',\\f,\Phi\cdot M)\subseteq K(\aa,\ff,\Phi)
		\end{equation}   
		On the other hand, $\Phi\cdot M\cdot M'$ is a matrix such that $(\aa)=(\ff)\cdot\Phi\cdot M\cdot M'$. By the same argument as above, we have   \begin{equation}
		K(\aa,\ff,\Phi\cdot M\cdot M')\subseteq K(\aa,\ff,\Phi\cdot M')\subseteq K(\aa,\ff,\Phi)
		\end{equation}
		The result now follows from the independence from the choice of the matrix $\Phi$, Proposition \ref{Pindmatrix}. 
	\end{proof}
	
	It now remains to prove that the disguised residual intersection does not depend on a choice of generators of $I$. For that we need two lemmas.
	
	\begin{lemma}\label{Lcycles}
		Let $R$ be a commutative ring,  $I = (f_1,\cdots ,f_r)$ an ideal, $1\leq i\leq r+1$, $f_0\in \Ann H_{i-1}(\ff;R)$ and $K_\bullet(f_0,\ff;R)=R\langle e_0,e_1,\cdots,e_r;\partial(e_i)=f_i\rangle$ the Koszul DG-Algebra. Then any cycle $z\in Z_i(f_0,\ff;R)$ can be uniquely written in the form $$z = e_0\wedge w + w'$$ where $w \in Z_{i-1}(\ff;R)$, $w'\in K_i(f_0,\ff;R)$ and $\partial(w') =-f_0w$. Conversely, for any $w \in Z_{i-1}(\ff;R)$ there exists $w' \in K_i(f_0,\ff;R)$ such that $e_0\wedge w + w' \in Z_i(f_0,\ff;R)$.
	\end{lemma}
	\begin{proof}
		Every element $z\in K_i(f_0,\ff;R)$ can be uniquely written in the form $z = e_0\wedge w + w'$ where $w\in K_{i-1}(\ff;R)$ and $w'\in K_i(f_0,\ff;R)$. If $z$ is a cycle, then \begin{equation}
		0=\partial(z) = f_0.w - e_0\wedge \partial(w') + \partial(w).
		\end{equation}
		Hence $\partial(w)=0$ and $\partial(w') = -f_0w$.
		
		For the converse, suppose that $w \in Z_{i-1}(\ff;R)$. Since $f_0\in \Ann H_{i-1}(\ff;R)$, $-f_0w$ is a boundary, that is, there is $w'\in K_i(\ff;R)$ with $\partial(w')= - f_0w$. Taking $z = e_0\wedge w+w'\in K_i(f_0,\ff;R)$, we have \begin{equation}
		\partial(z) = \partial(e_0\wedge w + w') = f_0w + e_0\wedge \partial(w) +\partial(w') =0 
		\end{equation}
		which proves the lemma.
	\end{proof} 


	\begin{lemma}\label{LgensI}
		Let  $f_0\in\cap_{i=\max\{0,r-s\}}^{r}\Ann H_i(\textbf{f},R)$ and $K_\bullet(f_0,\ff;R)=R\langle e_0, e_1,\cdots,e_r;\partial(e_i)=f_i\rangle$ the Koszul DG-Algebra. Then $$M=\begin{bmatrix}
		\textbf0\\\Phi
		\end{bmatrix}$$ satisfies $(\aa)=(f_0,\ff)\cdot M$, and $$K(\aa,\ff,\Phi) = K(\aa,(f_0,\ff),M).$$
	\end{lemma}
	\begin{proof}
		The assertion about the matrix is obvious. Let \begin{equation}
		\zeta_j = 0.e_0 +\sum\limits_{i=1}^rc_{ij}e_i,
		\end{equation} the $\zeta$'s corresponding to the representation matrix $M$. These elements can be viewed as the $\zeta$'s corresponding to the matrix $\Phi$. By Theorem \ref{Tgensdisguised2}, we have $K(\aa,(f_0,\ff), M) = <\Gamma_\bullet\cdot Z_\bullet>_{r+1} $.
		Hence, to construct a generator of $K(\aa,(f_0,\ff), M)$, we take $z \in Z_j(f_0,\ff;R)$, $r+1-s\leq j \leq r+1$ and $L_1\subseteq \{1,\cdots ,s\}$ with $|L_1| = r+1-j$. By Lemma \ref{Lcycles} $z = e_0\wedge w + w'$, where $w\in Z_{j-1}(\ff;R),w'\in K_j(\ff;R)$. Therefore \begin{equation}
		\zeta_{L_1} \wedge z = \zeta_{L_1}\wedge e_0 \wedge w + \zeta_{L_1} \wedge w'
		\end{equation}
		Since $\zeta_{L_1} \wedge w'$ is the wedge product of $r+1$ elements containing only $e_1,\cdots,e_r$, \begin{equation}
		\zeta_{L_1} \wedge w'=0.
		\end{equation}
		The product $\zeta_{L_1} \wedge w$ is the product of a cycle of degree $j-1$ with $(r+1-j)$  $\zeta$'s. Hence it gives an element in $K(\aa,\ff,\Phi)$. Therefore
		\begin{equation}
		K(\aa,(f_0,\ff),M)\subseteq K(\aa,\ff,\Phi).
		\end{equation} For the converse, let $w \in Z_j(\ff;R)$. By Lemma \ref{Lcycles}, there exists $w'\in K_{j+1}(\ff;R)$ such that \begin{equation}
		e_0 \wedge w +w' \in Z_{j+1}(\\a,(f_0,\ff);R).
		\end{equation} Let $L_1\subseteq\{1,\cdots ,s\}$ with $|L_1|=r-j$. We have that $e_0\wedge\zeta_{L_1} \wedge w=\zeta_{L_1}\wedge(e_0\wedge w + w')$. This shows that $\zeta_{L_1} \wedge w\in K(\aa,(f_0,\ff),M)$  
	\end{proof}

	We are now ready to prove the last part of the independence.
	
	\begin{prop}\label{Pindgensi}
		The disguised residual intersection does not depend on the choice of generators of $I$.
	\end{prop}
	\begin{proof}
		Let $(f_1,\cdots ,f_r)=(\ff),(f_1',\cdots ,f_t')=(\ff')$ be two sets of generators of $I$, $(a_1,\cdots ,a_s)=(\aa)$ a generating set for $\mathfrak a$, $\Phi$ and $\Phi'$ matrices such that $(\aa)=(\ff)\cdot\Phi$ and $(\aa')=(\ff')\cdot\Phi'$.
		By using repeatedly Lemma \ref{LgensI}, we have that \begin{equation}
		M=\begin{bmatrix}
		\textbf{0}\\\Phi
		\end{bmatrix}
		\end{equation} Satisfies $(\aa)=(\ff',\ff)\cdot M$, and $K(\aa,\ff,\Phi)=K(\aa,(\ff',\ff),M)$. Now, by Proposition \ref{Pindmatrix}, $K(\aa,(\ff',\ff),M)=K(\aa,(\ff',\ff),M')$ where \begin{equation}
		M'=\begin{bmatrix}
		\Phi'\\\textbf{0}
		\end{bmatrix}.
		\end{equation} Again, repeated applications of Lemma \ref{LgensI} gives us $K(\aa,(\ff',\ff),M')=K(\aa,\ff',\Phi')$.
	\end{proof}
	
	Now that we know the disguised residual intersection does not depend on any choice of generators or matrix $\Phi$, we introduce the following notation.
	
	\begin{defi}
		Let $R$ be a commutative  ring and $\mathfrak a \subseteq I$ be two finitely generated  ideals. We denote  the disguised residual intersection, $K(\aa,\ff,\Phi)$,  defined in Definition \ref{Ddisguised} by   $\Kitt(\mathfrak a,I)$.  

	\end{defi}

	This notation reminds that the disguised residual intersections are Koszul-Fitting ideals, based on Theorem \ref{Tgensdisguised2}.
	
	Lemmas \ref{Lcycles} and \ref{LgensI} provides some unexpected results about the codimension of the colon ideals and at the same time on the structure of the common annihilators of Koszul homologies. Both of these topics were mentioned as desirable in the works \cite{CHKV} and \cite{U1}.
	\begin{cor}\label{CannH}
		Let $R$ be a commutative ring, $I=(f_1,...,f_r)$ an ideal and $\mathfrak a=(a_1,...,a_s)\subseteq I$. Then $\Kitt(\mathfrak a,I) = \Kitt(\mathfrak a,I')$ for any ideal $I'$ satisfying 
		$$I \subseteq I'\subseteq \bigcap\limits_{\max\{0,r-s\}}^{r}\Ann H_i(\ff;R).$$ 
		 In particular, if $\Ht(I)=\Ht(I')$ (for instance when $R$ is Cohen-Macaulay), then  $(\mathfrak a:I)$ being  an $s$-residual intersection, implies that  $(\mathfrak a:I')$ is an $s$-residual intersection.  
	\end{cor}
\begin{proof}
	Let $(f_1',...,f_t')$ be a generating set for $I'$. The proof of Proposition \ref{Pindgensi} is applicable: it relies on of Lemma \ref{LgensI}, which works for elements $f_0 \in \cap_{\max\{0,r-s\}}^{r}\Ann H_i(\textbf{f},R)$. Therefore $\Kitt(\mathfrak a,I) = \Kitt(\mathfrak a,I')$.
	The second part of the statement follows from Theorem \ref{TaboutK}, stating that  these ideals have the same radical and \cite[Remark 1.5]{Hu1}.  
\end{proof}

\begin{rmk}
Although, in the above propositions, we have shown  that the structure of the disguised residual intersection $K=H_0(_0\mathcal Z^+_\bullet)$ is independent of the choice of generators, the other homologies of $_0\mathcal Z^+_\bullet$  are not independent of the choice of generators in general; see \cite[Theorem 4.4]{HN}
\end{rmk}
\subsection{Properties of Kitt ideals}
In this section we exhibit some basic properties of the $\Kitt$ ideals (Disguised residual intersections), based on structure Theorem \ref{Tgensdisguised2}. 

	\begin{prop}\label{Pzz1}
		Let $R$ be a commutative ring and keep  the same notation as in Theorem \ref{Tgensdisguised2}, we have  $$\langle\Gamma_\bullet \cdot \langle Z_1(\ff;R)\rangle\rangle_r={\rm Fitt}_0(I/\fa).$$ 
In particular, if the algebra of Koszul cycles of the Koszul complex $K_\bullet(\ff;R)$ is generated by cycles of degree one, then $\Kitt(\mathfrak a,I) = {\rm Fitt}_0(I/\mathfrak a).$
	\end{prop}
\begin{proof} Let $\Phi$ be an $r\times s$ matrix for which $(\aa)=(\ff)\cdot\Phi$ and $\Psi=(b_{ij})$ be a  syzygy matrix for the sequence $(\ff)$ which has $r$ rows. Then $Z_1 = Z_1(\ff;R)$ is generated by the elements $z_i=\sum_{i=1}^rb_{ij}e_i$. Therefore, $\langle\Gamma_\bullet \cdot \langle Z_1(\ff;R)\rangle\rangle_r$ is obtained by taking all the products of the form 
\begin{equation}\label{Egenz1}
	\zeta_{L_1}\wedge z_{L_2}, |L_1|+|L_2| = r.
	\end{equation}
	By elementary properties of the wedge product, a product as in (\ref{Egenz1}) is an $r\times r$ minor of the matrix $(\Phi|\Psi)$.  This matrix is the representation matrix of $I/\fa$ as it is obtained by taking the mapping-cone of the  following diagram
\begin{equation}
 \xymatrix{
 		R^t\ar[r]^{\Psi}& R^r\ar[r]& R \ar[r]& R/I\ar[r]& 0\\
		& R^s\ar[r]\ar[u]^{\Phi}& R \ar[r]\ar[u]^{=}& R/\fa\ar[r]\ar[u]& 0.
	}
	\end{equation}
	
		Thus $\langle\Gamma_\bullet \cdot \langle Z_1(\ff;R)\rangle\rangle_r={\rm I}_r(\Phi|\Psi)=\Fitt_0(I/\fa)$.
		
		\end{proof}
		
	In the next theorem,  we prove that the structure of $\Kitt(\mathfrak a,I)$ is encrypted on the Koszul homology algebra of the ideal $I$.  The main part of the proof is the following lemma. We fix some notations
	\begin{nota}\label{Dsgn}
		Let $n$ be an integer and  $I = \{i_1,\cdots , i_n\}$ be an ordered set. For any $J\subset I$ with $|J| = j$,  define $$
		\sgn(J\subset I)$$ to be the sign of the permutation that put the elements of $J$ on the first $j$ positions.
	\end{nota} 
	\begin{nota}\label{Dmatrix}
		Let $\Phi = (c_{ij})$ be a $r\times s$ matrix. 
		\begin{enumerate}
			\item Let $I\subset\{1,\cdots ,r\},J\subset \{1,\cdots ,s\}$ be two ordered subsets. Define $$\Phi^J_I$$ to be the submatrix with rows indexed by $I$ and columns indexed by $J$. If $I = \{1,\cdots ,r\}$ we supress the subscript and write $$\Phi^J.$$ We use an analogous notation if $J= \{1,\cdots ,s\}$
			\item Let $I\subset\{1,\cdots ,r\},J_1,J_2\subset \{1,\cdots ,s\}$ be three ordered subsets. Define $$\Phi^{J_1,J_2}_I$$ to be the submatrix with rows indexed by $I$, the first columns indexed by $J_1$ and the last columns indexed by $J_2$. 
		\end{enumerate}
	\end{nota}
	
	\begin{lemma}\label{LHtilde}
		Let $R$ be a commutative ring and keep the same notation as in Theorem \ref{Tgensdisguised2}, let $B_\bullet(\ff;R)$ be the ideal of Koszul boundaries. Then $$\langle\Gamma_\bullet \cdot B_\bullet\rangle_r = \mathfrak a.$$
	\end{lemma}
	
	\begin{proof}
		In the Koszul complex $K_\bullet(\ff;R)$, the module of boundaries of degree $k$ is generated by elements of the form $\partial(e_{L_2})$ where $|L_2| = k+1$.  For any $L_1\subseteq \{1,\cdots,s\}$ with $|L_1| = r-k$, we have \begin{equation}\label{Eboundary1}
		\zeta_{L_1}\wedge \partial(e_{L_2}) = \zeta_{L_1} \wedge (\sum\limits_{j\in L_2}\sgn(\{j\}\subseteq L_2)f_je_{L_2\setminus\{j\}})=\sum\limits_{j\in L_2}(\sgn(\{j\}\subseteq L_2)f_j\zeta_{L_1} \wedge e_{L_2\setminus\{j\}}).
		\end{equation}
		According to Proposition \ref{Peagonres}, the above equation  (\ref{Eboundary1}) can be  written as
		\begin{equation}\label{Eboundary2}
		\sum\limits_{j\in L_2}(\sgn(\{j\}\subseteq L_2)\sgn(L_2\setminus\{j\}\subseteq L)\det\Phi^{I}_{L\setminus(L_2\setminus\{j\})}f_j.
		\end{equation} 
		If we rearrange every determinant in a way such that the $j$-th row becomes the first one, (\ref{Eboundary2}) becomes
		\begin{equation}
		\sum\limits_{j\in L_2}(\sgn(\{j\}\subseteq L_2)\sgn(L_2\setminus\{j\}\subseteq L)\sgn(\{j\}\subseteq L\setminus(L_2\setminus\{j\}))\det\Phi^{L_1}_{\{j\},L\setminus L_2}f_j
		\end{equation}
		One can verifies, $(\sgn(\{j\}\subseteq L_2)\sgn(L_2\setminus\{j\}\subseteq L)\sgn(\{j\}\subseteq L\setminus(L_2\setminus\{j\}))$ does not depend on $j\in L_2$. Thus we can ignore this sign and take
		\begin{equation}\label{Eboundary3}
		\zeta_{L_1}\wedge \partial(e_{L_2}) =\sum\limits_{j\in L_2}\det\Phi^{L_1}_{\{j\},L\setminus L_2}f_j.
		\end{equation}
		If $j\notin L_2$ then $\det\Phi^{L_1}_{\{j\},L\setminus L_2}=0$,  since  $\Phi^{L_1}_{\{j\},L\setminus L_2}$ has a repeated row. Therefore (\ref{Eboundary3}) is equal to
		\begin{equation}
		\sum\limits_{j=1 }^r\det\Phi^{L_1}_{\{j\},L\setminus L_2}f_j.
		\end{equation}
		Now, we expand every determinant in this sum over the first row. Looking at each summand separately, we have \begin{equation}
		\det\Phi^{L_1}_{\{j\},L\setminus L_2}f_j=\sum\limits_{i \in L_1}\sgn(\{i\}\subset L_1)\det\Phi^{L_1\setminus \{i\}}_{L\setminus L_2}.c_{ji}f_j.
		\end{equation}
		Summing over all $j$, we get \begin{equation}
		\zeta_{L_1}\wedge \partial(e_{L_2}) =\sum\limits_{i\in L_1}\sgn(\{i\}\subset L_1)\det\Phi^{L_1\setminus \{i\}}_{L\setminus L_2}a_i.
		\end{equation}
		This shows that $<\Gamma_\bullet \cdot B_\bullet>_r \subseteq \mathfrak a$. 
		
		As to the other inclusion, we  consider the last boundary given by 
		\begin{equation}
		\partial(e_1\wedge\dots\wedge e_r)=\sum\limits_{i=1}^r (-1)^{i+1}f_ie_1\wedge\dots \hat{e_i}\dots\wedge e_r = :z
		\end{equation}
		Then, for any $1\leq j\leq s$, we have \begin{equation}
		\zeta_j\wedge z = \sum\limits_{i=1}^r c_{ij}f_i = a_i.
		\end{equation}
		
	\end{proof}
	
	\begin{thm}\label{TkittH} Let $R$ be a commutative ring and keep the same notation as in Theorem \ref{Tgensdisguised2} with $g=\grade(I)$. Let $\tilde{H}_\bullet$ is the sub-algebra of  $K_{\bullet}(\ff;R)$ generated by the representatives of Koszul homologies. Then 
	$$\Kitt(\mathfrak a,I)= \fa+\langle\Gamma_\bullet \cdot \tilde{H}_\bullet\rangle_r=\fa+\sum\limits_{i=\max\{0,r-s\}}^{r-g}\Gamma_{r-i}\cdot\tilde{H}_{i}.$$ 
		\end{thm}
		\begin{proof}
	According to  Theorem  \ref{Tgensdisguised2}, $\Kitt(\mathfrak a,I)= \langle\Gamma_\bullet \cdot Z_\bullet \rangle_r$. Since $ Z_\bullet= B_\bullet+\tilde{H} _\bullet$, we have  $\Kitt(\mathfrak a,I)= \langle\Gamma_\bullet \cdot B_\bullet \rangle_r+ \langle\Gamma_\bullet \cdot \tilde{H}_\bullet \rangle_r$. By Lemma \ref{LHtilde}, $\langle\Gamma_\bullet \cdot B_\bullet \rangle_r=\fa$ which yields the result.
		\end{proof}
	
\begin{rmk}\label{Rkitt}  The fact that $\fa\subseteq \Kitt(\fa, I)$ is not clear from the definition. So that concerning the  inclusions 
$$\Fitt_0(I/\fa)\subseteq \Kitt(\fa,I)\subseteq (\fa:I),$$  $\Kitt(\fa,I)$ is closer to $(\fa:I)$ than $\Fitt_0(I/\fa)$. Combining Theorem \ref{TkittH} with Proposition \ref{Pzz1}, 
one has:
 if an ideal $I=(\ff)$ in a commutative ring $R$ is such that the algebra of Koszul cycles of the Koszul complex $K_\bullet(\ff;R)$ is generated by cycles of degree one, then for any ideal $\fa\subseteq I$, $\fa\subseteq \Fitt_0(I/\fa)$.
\end{rmk}	
	Another  importance of  Theorem  \ref{TkittH} is that it connects the DG-algebra structure of Koszul homologies of $I$ to any colon ideal  $\fa:I$.   Even in the extremal cases where $I$ is complete intersection or an almost complete intersection this theorem provides highly non-trivial information about the structure of $J=\fa:I$.


	\begin{cor}\label{Ca.c.i.}
		Let $R$ be a commutative ring and  $I = (f_1,\cdots ,f_r)=(\ff)$ be an ideal such that the Koszul homology algebra $H_\bullet(\ff;R)$ is generated by elements of degree one. Then, for any finitely generated ideal  $\mathfrak{a}\subseteq I$, one has  $\Kitt(\mathfrak a,I)={\rm Fitt}_0(I/\mathfrak a) + \mathfrak a$. 
			In particular, this is the case when $(f_1,\cdots ,f_r)$ is an almost regular sequence (grade of $I$ is $r-1$).
		
		If $(f_1,\cdots ,f_r)$ is a  regular sequence then $\Kitt(\mathfrak a,I)=I_r(\Phi) + \mathfrak a$ where $\Phi$ is an $r\times s$ matrix satisfying $\aa=\ff\cdot \Phi$. 
	\end{cor}
	\begin{proof}
	Just notice that in the case of complete intersection $\tilde{H}_\bullet$ in Theorem \ref{TkittH} in concentrated in degree zero that is $\tilde{H}_\bullet=R$; so that 
	$$\langle\Gamma_\bullet \cdot \tilde{H}_\bullet\rangle_r=\langle\Gamma_\bullet \cdot R\rangle_r=I_r(\Phi)$$
	\end{proof}
It's clear that the construction that Kitt ideals commutes with localization. The case of specialization modulo a regular sequence $\alpha=(\alpha_1,\dots,\alpha_g)\subset\mathfrak a$ is more subtle and will be fixed in the next proposition.  The following lemma is necessary for the proof. 

	
	\begin{lemma}\label{Lhomology}
	 Let $R$ be a commutative  ring, $I = (f_1,\cdots ,f_r)=(\ff)$. Let  $f_0\in I$ be  a $R$-regular element and consider the Koszul complex $K_\bullet=R\langle e_0,\cdots ,e_r: \partial(e_i)=f_i\rangle$.Then there is an isomorphism $$H_i(f_0,\ff;R)\rightarrow H_i(\ff;R/f_0)$$ given by the map $$e_0\wedge w + w' \rightarrow {\tilde w'}$$
	 where $w \in Z_{i-1}(\ff;R)$, $w'\in K_i(f_0,\ff;R)$ and $\partial(w') =-f_0w$
	 
	\end{lemma}
	\begin{proof}
		  The proof is essentially the one of Lemma \ref{Lcycles}; see also   \cite[Proposition 1.6.12(c)]{BH}.
	\end{proof}
	
	We now prove the  last theorem in this section, showing that the disguised residual intersection specializes modulo a regular sequence contained in $\mathfrak a$.
	
		\begin{thm}\label{Tkittspecial}
		Let $R$ be a commutative  ring $\mathfrak a\subseteq I$ finitely generated ideals  and $f_0\in \mathfrak a$ an $R$-regular element. Then $\Kitt(\mathfrak a, I)/(f_0)=\Kitt(\mathfrak a/(f_0),I/(f_0))$ 
	\end{thm}
	\begin{proof}
		First, we notice that $f_0\in \Kitt(\mathfrak a, I)$ by Theorem \ref{TkittH}. Also for an element $r \in R$, put  $\widetilde{r}$ to denote the image of $r$  via the projection homomorphism $R\rightarrow R/(f_0)$. 
		
		Fix generators $(f_1,\cdots ,f_r)$ of $I$, $(a_1,\cdots ,a_s)$ of $\mathfrak a$, a matrix $\Phi= (c_{ij})$ such that $(\aa)=(\ff)\cdot\Phi$, and let $\zeta_j = \sum_{i=1}^rc_{ij}e_i \in K_1(f_1,\cdots ,f_r;R)$. It's clear that $\widetilde I =(\widetilde \ff)$, $\widetilde{\mathfrak a} =(\widetilde \aa)$ and $\widetilde\Phi$ satisfies $(\widetilde\aa)=(\widetilde\ff)\cdot\widetilde\Phi$. Setting $\widetilde{\zeta_j} = \sum_{i=1}^r \widetilde{c_{ij}}e_i$ and $\widetilde{\Gamma}_\bullet=\frac{R}{(f_0)}[\widetilde{\zeta_1},\cdots,\widetilde{\zeta_s}]\subseteq K_\bullet(\widetilde{\ff};R/(f_0))$, we have, by Theorem \ref{Tgensdisguised2},
		\begin{equation}
	\Kitt(\frac{\mathfrak a}{(f_0)},\frac{I}{(f_0)})  = <\widetilde{\Gamma}_\bullet\cdot Z_\bullet(\widetilde\ff;R/(f_0))>_r.
		\end{equation}    
		Let $z\in Z_j(\widetilde\ff;R/(f_0))$, $0\leq j \leq r$ and $L_1\subseteq \{1,\cdots ,s\}$ such that $|L_1| = r-j$. We need to prove that $ \widetilde\zeta_{L_1}\wedge z$ is the specialization of some elements in $\Kitt(\mathfrak a,I).$ 
		 By Lemma \ref{Lhomology}, there is a cycle $c = e_0\wedge w + w'\in Z_j(f_0,\ff;R)$ such that $z=\widetilde{w'}$ in $ H_j(\widetilde\ff;R/(f_0))$. 
		 
		 According to Theorem \ref{TkittH}, it suffices to prove that $\zeta_{L_1}\wedge w'$ is an element in $\Kitt(\mathfrak a, I).$
 Since $f_0 \in \mathfrak a$ there exist $\alpha_i\in R$ such that
	\begin{equation}
	f_0 = \sum\limits_{i = 1}^s \alpha_ia_i.
	\end{equation}
	 
 Hence $e_0 - \sum_{i = 1}^s \alpha_i\zeta_i\in Z_1(f_0,\ff;R)$. Therefore, Theorem \ref{Tgensdisguised2}  implies that 
 $$\zeta_{L_1} \wedge (e_0 -\sum\limits_{i=1}^s \alpha_i\zeta_i)\wedge c \in \Kitt(\mathfrak a, I).$$
 On the other hand, 	
 \begin{equation}\label{Eexpression}
 \zeta_{L_1} \wedge (e_0 -\sum\limits_{i=1}^s \alpha_i\zeta_i)\wedge c = \zeta_{L_1}\wedge(-w'\wedge e_0 - \sum\limits_{i=1}^s \alpha_i\zeta_i\wedge w\wedge e_0 - \sum\limits_{i=1}^s \alpha_i\zeta_i \wedge w' ) 
 \end{equation}
 
  On the summands on the right side, we have
 
 \begin{itemize}
 	\item $\zeta_{L_1}\wedge  \sum_{i=1}^s \alpha_i\zeta_i \wedge w'$ which is zero,  since it's a wedge product of $r+1$ elements involving only $e_1,\cdots,e_r.$
 		\item $\zeta_{L_1} \wedge \sum_{i=1}^s \alpha_i\zeta_i \wedge w$ which gives us a generator of $\Kitt(\mathfrak a, I)$ by Theorem \ref{Tgensdisguised2}. 
 \end{itemize}
It then follows that $\zeta_{L_1}\wedge w'$ is an element  in $\Kitt(\mathfrak a, I)$ as desired.		
	\end{proof}

\section{Applications and Corollaries}
For the first applications of the facts developed in the previous sections, we will present the following theorem  which  proves Conjecture \ref{conj} to a certain extent.
\begin{thm}\label{Tmain}
    Let $R$ be a Cohen-Macaulay ring and  $I$ be an ideal of height $g\geq 2$ which satisfies  $\SD_1$ condition. Then any algebraic $s$-residual intersection  $J=\fa:I$ coincides with the disguised residual intersection.
\end{thm}
\begin{proof}
According to Theorem \ref{TaboutK}(i), $K\subseteq J$. So that to prove the equality, without loss of generality, we may assume that $R$ is a complete Cohen-Macaulay local ring and hence possesses a canonical module. Moreover $K=\Kitt(\fa,I)$ by the structure Theorems in section 4. 

Let $\alpha=\alpha_1,\cdots,\alpha_{g-2}\subseteq \fa$ be a regular sequence which is a part of $s$ generators of $\fa$. $I/\alpha$ still satisfies the $\SD_1$ condition \cite[Proposition 4.1]{CNT} and $(\frac{\fa}{\alpha}:\frac{I}{\alpha})=\frac{J}{\alpha}$ is an $(s-g+2)$-residual intersection. So that  \cite[Theorem 4.5]{CNT}( Theorem\ref{TaboutK}(4)) implies that  
$$\Kitt(\frac{\fa}{\alpha},\frac{I}{\alpha})=(\frac{\fa}{\alpha}:\frac{I}{\alpha})=\frac{J}{\alpha}.$$
Now by Theorem \ref{Tkittspecial}, we have  $$\Kitt(\frac{\fa}{\alpha},\frac{I}{\alpha})=\frac{\Kitt(\fa,I)}{\alpha}$$
which proves the theorem.
\end{proof}
Theorem \ref{Tmain} has several consequences. Indeed all of the known properties of disguised residual intersections are the properties of algebraic residual intersections if the ideal $I$ satisfies the $\SD_1$ condition. Among all, there are the Cohen-Macaulayness, the Castelnuovo-Mumford regularity and  the type of  algebraic residual intersections. We also have the following important property about the behaviors of the Hilbert functions. 
\begin{cor}\label{CHilbert} Let $R$ be a CM standard graded ring over an Artinian local ring $R_0$.   Suppose that $I$ satisfies $SD_1$ condition.  Then for any  $s$-residual intersection  $J=(\fa:I)$, 
  the Hilbert function of  $R/J$ depends on the ideal $I$ and merely  on the \textbf{ degrees} of the generators of $\fa$.
\end{cor}
\begin{proof}
The fact has been already proved for disguised residual intersections in  \cite[Proposition 3.1]{HN}. Due to Theorem \ref{Tmain} the disguised residual intersection is the same as the algebraic residual intersection for ideals with SD$_1$.
\end{proof}
Besides the coincidence of disguised and algebraic residual intersections, we have the structure of the generators of the disguised residual intersections by Theorem \ref{TkittH}. This fact leads to some important classification of residual intersections.
We mention one immediate corollary here. 
\begin{cor} Let $R$ be a Cohen-Macaulay ring and  $I$ be a complete intersection ideal generated by $\ff=f_1,\cdots, f_r$. Suppose that  $J=\fa:I$ is an algebraic $s$-residual intersection of $I$. Then $J=I_r(\Phi) + \mathfrak a$ 
	  where $\Phi$ is   an $r\times s$ matrix satisfying $\aa=\ff\cdot \Phi$.
  
\end{cor}
\begin{proof}
Complete intersections are $\SD_1$ obviously so that we have $J=\Kitt(\fa,I)$ by Theorem \ref{Tmain}. The result now follows from Corollary \ref{Ca.c.i.}.
\end{proof}
This corollary  provides  generalizations to \cite[Theorem 4.8]{BKM} and  \cite[Theorem 5.9(i)]{HU}.  The former works for geometric residual intersections and the latter needs $R$ to be Gorenstein domain, the proof of the latter in turn appeals a result of Deconcini and E. Strickland \cite{DS} to determine the structure of residual intersection ideal $J$.

As far as we know, there is no  structural result as above  for residual intersections of almost complete intersections. We have the following.


\begin{cor} Let $R$ be a Cohen-Macaulay ring and  $I$ be an almost complete intersection ideal which is Cohen-Macaulay. Let $J=\fa:I$ be an algebraic $s$-residual intersection of $I$. Then $J=\Fitt_0(I/\fa)+\fa$. 
\end{cor}
\begin{proof}
Since almost complete intersection CM ideals are SCM,  this is another  consequence of  Theorem \ref{Tmain} and  Corollary \ref{Ca.c.i.}.
\end{proof}
The DG-algebra structure of $I$ has some non-trivial impacts on the structure of residual intersections. For instance we have
\begin{cor}
Let $R$ be a Cohen-Macaulay ring and  $I$ be a perfect ideal of height $2$. Let $J=\fa:I$ be an algebraic $s$-residual intersection of $I$. Then $J=\Fitt_0(I/\fa)$.
\end{cor}
\begin{proof}
A result of Avramov-Herzog \cite[Proof of Theorem 2.1(e)]{AH} shows that for perfect ideals of height $2$ the algebra of cycles of Koszul is generated in degree $1$. So that the result follows from Theorem \ref{Tmain} and  Proposition \ref{Pzz1}.
\end{proof}
There are other ways to prove the above result, see for example \cite{Hu},\cite{KU} or \cite[Theorem 1.1]{CEU}.

We can also study the common annihilator of Koszul homologies using Corollary \ref{CannH} without any presence of Sliding depth hypotheses. 
\begin{cor}
		Let $R$ be a Cohen-Macaulay local ring and   $\mathfrak a\subseteq I=(f_1,\cdots,f_r)=(\ff)$ ideals of $R$. Let $J=\fa:I$ be an $s$-residual intersection of $I$.
		 Suppose in addition that either $\fa$ is NOT generated by an analytic independent set of generators or else $\Ht(J)\geq s+1$.
		 Then   $$  \bigcap\limits_{\max\{0,r-s\}}^{r}\Ann H_i(\ff;R)\subseteq \bar{\fa}$$
		 where $\bar{\fa}$ is the integral closure of $\fa$.
\end{cor}
\begin{proof}
The proof is a consequence of Corollary \ref{CannH} applying to a nice result of Huneke and(or) Ulrich \cite[Proposition 3]{U1}. 
\end{proof}
Although the $\SD_1$ condition appears in the Theorem \ref{Tmain} and hence we need it in all of the corollaries, it is not in an essential way. If one can show that the disguised residual intersection and the algebraic residual intersection coincide for any nice class of ideals of small height then the techniques above will provide  equality $\Kitt(\fa,I)=(\fa:I)$ quite generally. 
However, sliding depth conditions are necessary if one seeks Cohen-Macaulay residual intersections \cite{U}. On the other hand, we guess that to prove $\Kitt(\fa,I)=(\fa:I)$ one can totally forget it.
\begin{conj}\label{CojK=J}
Let $R$ be a Cohen-Macaulay ring then $\Kitt(\fa,I)=(\fa:I)$ whenever $J=\fa:I$ is an algebraic $s$-residual intersection.
\end{conj}
As corollaries, one can remove SD$_1$ condition from the results in this section.
By the way one may not expect that $\Kitt(\fa,I)=(\fa:I)$ for any pair of ideals $\fa$ and $I$.

\begin{exa}\label{Example}Let $R=\mathbb{Z}_3[x,y,z,t]$, $I=(x^2,y^2,xy,xt-yz)$ and $\fa=(x^4,y^4,x^2y^2)$. 
Then $J=\fa:I$ has height $2$ so that it is not a $3$-residual intersection. One can check that the Koszul homology algebra of $I$ is generated in degree $1$; so that  $\Kitt(\fa,I)=\fa+\Fitt_0(I/\fa)$ by Corollary \ref{Ca.c.i.}. A Macaulay verification shows that  $\Kitt(\fa,I)\neq J$.

\end{exa}

In the following proposition, we prove Conjeture \ref{CojK=J} in the case where $s\leq g+1$. 

\begin{prop}\label{Ps<g+1}
Let $R$ be a Cohen-Macaulay ring, $I\subset R$ an ideal with $\Ht(I) = g$ and $J = \mathfrak a:I$  an $s$-residual intersection.  If $s\leq g+1$  then $\Kitt(\mathfrak a:I) = J$.
\end{prop}
\begin{proof}
Let $a_1,\cdots,a_s$ be a set of generators of $\mathfrak a$. By Proposition \ref{Pindgensa}, we can suppose that $a_1,\cdots,a_g$ is a regular sequence. By Theorem \ref{Tmain}, we can mod out this regular sequence and hence we may and do suppose that $g=\Ht(I) = 0.$ Let $f_1,\cdots,f_r$ be a set of generators of $I$ and let $K_\bullet = R<e_1,\cdots,e_r;\partial(e_i)=f_i>$ be the Koszul complex of $\ff$.

If $s = 0$, then $\mathfrak a = 0$ and we must show that $\Kitt((0),I) = (0:I)$. In this case, $\Gamma_\bullet = R$ and we have
$$\Kitt(\mathfrak a, I) = \Gamma_0 \cdot Z_r(\ff;R)= R\cdot(0:I) = (0:I)=J.$$

If $s = 1$, then by Proposition \ref{Pindgensi} we may suppose that $\mathfrak a = (f_1)$. In this case, $\Gamma_\bullet = R\oplus R\cdot e_1$ and then $$\Kitt(\mathfrak a, I) = \Gamma_0 \cdot Z_r(\ff;R)+\Gamma_1\cdot Z_{r-1}.$$

Write $\hat{e_i}$ for $e_1\wedge\cdots\wedge\hat{e_i}\wedge\cdots\wedge e_r$. Then $c = \sum\limits_{i=1}^ra_i\hat{e_i}\in Z_{r-1}(\ff;R)$ if and only if 
$$I_2\begin{bmatrix} a_1 & a_2 & \cdots & a_r\\f_1 & f_2 &\cdots & f_r
\end{bmatrix} = 0.$$

Moreover the element of $\Kitt(\mathfrak a, I)$ produced by $c$ is $e_1\wedge c = a_1$. Let $\alpha \in J$. To show that $\alpha \in \Kitt(\mathfrak a, I)$, one needs to show that there are $\alpha_2,\cdots,\alpha_r\in R$ such that 
\begin{equation}\label{EI2}
    I_2\begin{bmatrix} \alpha & \alpha_2 & \cdots & \alpha_n\\f_1 & f_2 &\cdots & f_r
\end{bmatrix} = 0.
\end{equation}

Since $\alpha \in J=(f_1):I$, there are $\alpha_2,\cdots,\alpha_r$ such that 
\begin{equation}\label{I21}
    \alpha f_i= \alpha_if_1.
\end{equation} 
If one  shows that $\alpha_i f_j = \alpha_j f_i$ for $i,j \geq 2$, then $(\alpha_1,\cdots,\alpha_r)$ satisfies (\ref{EI2}); so that $\alpha\in \Kitt(\mathfrak a, I)$. If we  suppose in addition  that $\alpha$  is $R$-regular. Then the desired equality is equivalent to 
\begin{equation}\label{I22}
    \alpha_i f_j\alpha = \alpha_j f_i \alpha,
\end{equation}  which holds  due to (\ref{I21}).

Thence any $R$-regular element in $J$ belongs to $\Kitt(\mathfrak a, I)$.
Now, let $\beta \in J$ be an arbitrary element. Since $J$ is a $1$-residual intersection,  $\Ht(J) \geq 1$. Since $R$ is Cohen-Macaulay, $\grade(J)\geq 1$. Therefore there exists $\alpha \in J$ which is  $R$-regular. Let $\alpha_i, 2\leq i\leq r$ satisfy (\ref{EI2}).  Let $\beta_2,\cdots,\beta_r$ be such that 
\begin{equation}\label{I23}
    \beta f_i= \beta_if_1
\end{equation}

The equations (\ref{I21}) and (\ref{I23}) implies that \begin{equation}\label{I24}
    (\alpha + \beta)f_i = f_1(\alpha_i+\beta_i).
\end{equation}
By applying (\ref{EI2}), we have   \begin{equation}\label{I25}
     f_i (\alpha_j + \beta_j) -  f_j(\alpha_i+\beta _i) = (f_i\alpha_j - f_j\alpha_i) + (f_i\beta_j-\beta_jf_i)=f_i\beta_j-\beta_jf_i.
\end{equation}

According to   (\ref{I24}),  
$(f_i (\alpha_j + \beta_j) -  f_j(\alpha_i+\beta _i))(\alpha+\beta)=0$. This in conjunction with (\ref{I25}) implies that 
\begin{equation}\label{I26}
     (f_i\beta_j-\beta_jf_i)(\alpha+\beta)= 0.
\end{equation}
Now, (\ref{I23}) implies that
\begin{equation}
    (f_i\beta_j-\beta_jf_i)(\beta)= 0.
\end{equation}
Subtracting the last two equations we get $$(f_i\beta_j-\beta_jf_i)(\alpha)= 0.$$
Which implies, by the regularity of $\alpha$, that 
\begin{equation}\label{I27}
    f_i\beta_j-\beta_jf_i = 0. 
\end{equation}
Therefore 
$$I_2\begin{bmatrix} \beta & \beta_2 & \cdots & \beta_r\\f_1 & f_2 &\cdots & f_r
\end{bmatrix} = 0$$
which implies that $\beta\in \Kitt(\fa,I)$.

\end{proof}
\begin{rmk}
One can see from the proof of Proposition \ref{Ps<g+1} that the Cohen-Macaulay assumption on $R$ is only needed to guarantee the existence of a regular sequence in $J$.  In other words, if in the  definition of residual intersection, Definition \ref{Dresidual}, one replaces $\Ht(J)\geq s$ with $\grade(J)\geq s$. The result of Proposition \ref{Ps<g+1} holds for any Noetherian ring.

Even the case where $s=g$, the above result for $J=\fa:I$ is more general than the known linkage theory, as here, $R$ is not Gorenstein and $I$ is not unmixed, necessarily. However,   Proposition \ref{Ps<g+1} determines the set of generators of $J=\fa:I$. 

Another importance of Conjecture \ref{CojK=J} or  Theorem \ref{Tmain} and Proposition \ref{Ps<g+1} is that: from the structure of colon ideal, $J=\fa:I$, it is not clear that $J$ can be specialized, particularly from a generic choice of $\fa$ to a general choice of $\fa$. However $\Kitt(\fa,I)=\langle\Gamma_{\bullet}\cdot Z_{\bullet}\rangle_r$ specializes naturally.  
\end{rmk}

One can also detect Cohen-Macaulay residual intersections under very slight conditions. 
\begin{prop}\label{Presolution}
Let $R$ be a Cohen-Macaulay ring, $I\subset R$ an ideal with $\Ht(I) = g$ and $J = \mathfrak a:I$  an $s$-residual intersection with $s\leq g+1$. Then $_0\mathcal{Z}^+_{\bullet}$ resolves $R/J$. 

More precisely, let $I=(f_1,\cdots,f_r)$,  $Z_{i}=Z_{i}(\ff,R)$  the Koszul cycles and $Z_{j}^+$ mean a direct sum of copies of Koszul cycles $Z_i$ for $i\geq j$, then 
\begin{enumerate}
\item  If $s=g$,  there exists an exact complex $0\to F_g\to\cdots\to F_2\to Z_{r-g}^+\to R\to R/J\to 0$ wherein $F_i$'s are free $R$-modules.
    \item If $s=g+1$,  there exists an exact complex $0\to F_{g+1}\to\cdots\to F_3\to Z_{r-g}^+\to Z_{r-g-1}^+\to R\to R/J\to 0$ wherein $F_i$'s are free $R$-modules.
    \end{enumerate}
    In particular,
    \begin{itemize}
        \item If $s=g$, $R/J$ is Cohen-Macaulay if and only if $\depth(Z_{r-g})\geq d-g+1$. In this case $\depth(Z_{r-g})= d-g+1$
    \item If $s=g+1$, $R/J$ is Cohen-Macaulay if  $\depth(Z_{r-g})\geq d-g+1$ and  $\depth(Z_{r-g-1})\geq d-g$.
\end{itemize}
\end{prop}
\begin{proof}
Parts (1) and (2) follows from the construction of $\mathcal{Z}^{'}_{\bullet}$ complex in \cite[Page 6375]{Ha} and \cite[Corollary 2.9(c)]{Ha}. The statements about Cohen-Macaulay property follow from a usual diagram chasing (or spectral sequence) applying  to the complexes in the first part. We notice that one cannot deduce these Cohen-Macaulay properties appealing the standard sequence $0\to B_{r-g}\to Z_{r-g}\to H_{r-g}\to 0.$ 
\end{proof}
B. Ulrich \cite{U} defines the Artin-Nagata property, AN$_s$, for the  ideal $I$  in a Cohen-Macaulay ring $R$, if every $i$-residual intersection of $I$ is Cohen-Macaulay for any $i\leq s$. As a consequence of Conjecture \ref{conj},  \cite[Theorem 2.11]{Ha} implies that AN$_s$ is equivalent to SDC$_1$ condition at level $\min\{s-g,r-g\}$, Definition \ref{dsd}.  Nevertheless, we have the following corollary
\begin{cor}
        Let $R$ be a Cohen-Macaulay local ring of dimension $d$, $I\subset R$ an ideal  generated by $r$ elements with $\Ht(I) = g$ and $g\leq s\leq g+1$. Then the following are equivalent
        \begin{itemize}
            \item[(i)] $I$ satisfies AN$_s$.
            \item[(ii)] For any $i\leq s$ there exists an  $i$-residual intersection of $I$ which is Cohen-Macaulay. 
            \item[(iii)] $I$ satisfies SDC$_1$ condition at level $\min\{s-g,r-g\}$.
        \end{itemize}
         
        \end{cor}
        \begin{proof}
        $(i)\Rightarrow (ii)$  trivially.  $(iii)\Rightarrow (ii)$, according to  Proposition \ref{Ps<g+1} and \cite[Theorem 2.11]{Ha}.
        For $(ii)\Rightarrow (iii)$, we use the resolutions in Proposition \ref{Presolution}. Since $g$-residual intersections are CM, Proposition \ref{Presolution}(1) implies that $\depth(Z_{r-g}\geq d-g+1$. Now having  $\depth(Z_{r-g})\geq d-g+1$ and $\depth(R/J)\geq d-g-1$, Proposition \ref{Presolution}(2) implies that  $\depth(Z_{r-g-1})\geq d-g$.
        \end{proof}

\section{Acknowledgments}
The first author was supported by a Ph.D. scholarship from CAPES-Brazil and of FAPERJ-Brazil. The second author was supported by   the grant ``bolsa de produtividade 301407/2016-9" from  CNPq-Brazil. The authors  would like to thank CAPES and CNPq for their support. 
They also thank Marc Chardin and Jose Naeliton for useful discussions.


\end{document}